\documentclass[11pt, a4wide]{amsart}  

\usepackage{graphicx}
\usepackage{amssymb}
\usepackage{amsmath}
\usepackage{tikz-cd}

\newcommand{\ts}{\hspace{0.5pt}}

\newcommand{\euler}{\mathrm{e}}
\newcommand{\imi}{\mathrm{i}}

\newcommand{\CC}{\mathbb{C}\ts}
\newcommand{\RR}{\mathbb{R}\ts}

\newcommand{\ZZ}{\mathbb{Z}}
\newcommand{\NN}{\mathbb{N}}
\newcommand{\TT}{\mathbb{T}}

\newcommand{\oplam}{\mbox{\Large $\curlywedge$}}

\newcommand{\LL}{\mathcal{L}}

\newcommand{\dd}{\,{\rm d}}

\newcommand{\MM}{\mathcal{M}(G)}
\newcommand{\MCV}{\mathcal{M}_{C,V}(G)}
\newcommand{\MTB}{\mathcal{M}^{\infty}(G)}

\newcommand{\cM}{\mathcal{M}}

\newcommand{\Oomega}{(\varOmega,\alpha)}

\newcommand{\Ttheta}{(\varTheta,\beta)}

\newcommand{\LO}{L^2 (\varOmega,m)}

\newcommand{\Ghat}{\widehat{G}}

\newcommand{\supp}{\mbox{supp}}

\newcommand{\cL}{\mathcal L}
\newcommand{\cp}{(G,H,\cL)}

\newtheorem{lemma}{Lemma}
\newtheorem{theorem}[lemma]{Theorem}
\newtheorem{proposition}[lemma]{Proposition}
\newtheorem{corollary}[lemma]{Corollary}
\newtheorem{definition}[lemma]{Definition}
\newtheorem{example}[lemma]{Example}
\newtheorem{remark}[lemma]{Remark}

\numberwithin{lemma}{section}

\begin{document}


\title[]{Modulated crystals and almost periodic measures
}

\dedicatory{We dedicate this work to Michael Baake on the occasion of his $60^{th}$ birthday.}

\author{Jeong-Yup Lee}
\address{Department of Mathematics Education, Catholic Kwandong University, Gangneung, Gangwon 25601, Korea
and KIAS, 85 Hoegiro, Dongdaemun-gu, Seoul, 02455, Korea}
\email{jylee@cku.ac.kr}

\author{Daniel Lenz}
\address{Fakult\"at f\"ur Mathematik und Informatik, Friedrich Schiller Universit\"at Jena, 07737 Jena, Germany}
\email{dlenz.lenz@uni-jena.de}

\author{Christoph Richard}
\address{Department f\"{u}r Mathematik, Friedrich-Alexander-Universit\"{a}t Erlangen-N\"{u}rnberg, Cauerstrasse 11, 91058 Erlangen, Germany}
\email{christoph.richard@fau.de}

\author{Bernd Sing}
\address{Department of Mathematics, University of the West Indies, Cave Hill, P.O.\ Box 64, Bridgetown, BB11000, Barbados} 

\email{bernd.sing@cavehill.uwi.edu}

\author{Nicolae Strungaru}
\address{Department of Mathematical Sciences, MacEwan University,
10700 “ 104 Avenue, Edmonton, AB, T5J 4S2,
and Institute of Mathematics ``Simon Stoilow'', Bucharest, Romania}
\email{strungarun@macewan.ca}

\renewcommand{\thefootnote}{\fnsymbol{footnote}} 
\renewcommand{\thefootnote}{\arabic{footnote}}


\maketitle

\begin{abstract}
Modulated crystals and quasicrystals can simultaneously be described as modulated quasicrystals, a class of point sets introduced by de Bruijn in 1987. With appropriate modulation functions, modulated quasicrystals themselves constitute a substantial subclass of strongly almost periodic point measures. We re-analyse these structures using methods from modern mathematical diffraction theory, thereby providing a coherent view over that class. Similarly to de Bruijn's analysis, we find stability with respect to almost periodic modulations.
\end{abstract}

\maketitle

\section{Introduction}

In 1964, Brouns et al.~performed an X-ray diffraction experiment on a  washing soda crystal, which was expected to behave like a perfect crystal. Surprisingly, a slight anomaly in its diffraction pattern was found \cite{BVW64}. It was later suggested to describe such anomalies by modulated crystals \cite{W74,JJ77}. Over the years, the crystallography of such structures has been worked out and has found many applications including protein crystallography \cite{Lea07}. Numerous important contributions to the field have been made by Ted Janssen \cite{Sou18}.
Recent overviews of modulated structures and of other types of aperiodic crystals are Janssen \cite{J12} and Janner and Janssen \cite{JJ14}. See also the monographs by Janssen et al. \cite{JCB07} and by van Smaalen \cite{vS07} for a detailed discussion of aperiodic crystals  from physical and crystallographic perspectives.

\smallskip

On the mathematical side, Bombieri and Taylor \cite{BT87} suggested studying modulated lattices in 1985, as they appeared to have properties similar to quasicrystals, a fundamentally different type of aperiodic crystal that had been discovered shortly before. Models for quasicrystals were obtained by the so-called cut-and-project construction \cite{KN,DK}. De Bruijn introduced modulated quasicrystals in Euclidean space \cite{dB87}, a class of point sets that is nowadays subsumed by so-called deformed weighted model sets. Whereas this class comprises weighted cut-and-project sets, de Bruijn showed that it also includes certain modulated lattices \cite[Sec.~5,6]{dB87}.
De Bruijn's Fourier analysis relied on a particular class of smooth weight functions of unbounded support, which cannot readily be extended beyond Euclidean space.

\smallskip

Deformed weighted model sets have further been studied by Hof \cite{Hof3} and by Bernuau and Duneau \cite{BD}, who coined the name.
In the meantime, it was realized that cut-and-projects sets are in fact model sets as introduced earlier by Meyer \cite{Meyer,Moody-old,Moody-model sets}. Also Euclidean space was generalized to a locally compact abelian group, and dynamical systems techniques were introduced as a powerful tool \cite{Ro96,Martin2}. In that setting, deformed weighted model sets (in fact much more generally deformed weighted Delone sets) were analyzed by Baake and Lenz \cite{BL2}. Modulated structures were studied by Sing \cite{SW06,WS07,S08}, by reconstructing the internal space of the underlying deformed model set from the diffraction pattern of the modulated structure. Let us stress that there is no universally agreed definition of modulation or deformation at present.

\smallskip

As mathematical diffraction theory has now reached a certain maturity, see e.g.~\cite{BG2,MoSt,LenStr16,RS17} for recent expositions of various parts, it seems appropriate to re-analyze modulated crystals within that general framework. Whereas mean almost periodicity of the underlying point set emerges to be characteristic of pure point diffraction in general, modulated crystals are examples of more restrictive strongly almost periodic point sets. For that reason, we adopt the setting of strongly almost periodic (SAP) measures on locally compact abelian groups in our article \cite{GdeL,MoSt}. On a conceptual level, diffraction has recently been studied for more general weakly almost periodic measures \cite{LenStr16}. In particular it is known that such measures are pure point diffractive, as their autocorrelation is strongly almost periodic \cite[Thm.~7.5(a)]{LenStr16}.
 Also, their diffraction amplitudes can be computed as the squared modulus of their Fourier--Bohr coefficients \cite[Thm.~7.5(b)]{LenStr16}.
 
 \smallskip

Although our motivation is modulations of lattices and of general ideal crystals, we will profit from the considerably more general setting of deformed weighted model sets.
Hence in this article, we discuss  deformed weighted model sets as a substantial subclass of SAP measures.
We will prove that deformed weighted model sets with continuous and compactly supported weight and deformation functions are strongly almost periodic, thereby extending older results for weighted model sets \cite{BM,LR,Ric,NS11,NS12}.
We will then deform such structures using continuous almost periodic modulations.
We will prove that such modulations of deformed weighted model sets stay within the class of deformed weighted model sets, which is in line with de Bruijn's analysis~\cite{dB87}.
We will also show that, in Euclidean space, the class of deformed weighted model sets in fact coincides with the class of modulated weighted model sets. In retrospect, this justifies de Bruijn's terminology. Our diffraction analysis relies on dynamical systems techniques based on the so-called torus parametrization \cite{BHP,Ro96,Martin2,LR,KR,KR17}. In particular, it yields short alternative proofs of previous results.

\smallskip

Our examples also shed some light on Lagarias' question  \cite{Lag}: Which Delone Dirac combs are strongly almost periodic? Whereas finite local complexity enforces crystallinity \cite[Cor.~5.6]{KL13}, deformed weighted model sets provide many examples of infinite local complexity, e.g.~almost periodic modulations of lattices or of ideal crystals.

\smallskip

Let us briefly mention two lines of problems within that context.
It would be interesting to classify SAP measures. But at present we do not even know of a simple characterization of SAP Dirac combs. Note that such combs need not be modulated crystals: think for example of the incommensurate structure $\mathbb Z \cup \alpha\mathbb Z$ where $\alpha$ is irrational \cite[Ex.~9.6]{BG2}. Whereas this point set is not uniformly discrete, it is possible to define uniformly discrete variants \cite[Ex.~9.7]{BG2} and \cite[Ch.~4]{vS07}. There are other examples on the line, arising as factors of the Kronecker flow on the two-dimensional torus \cite[Sec.~6]{KL13}, which fail to be modulations of the integer lattice.

Another aspect concerns modulation functions beyond almost periodic ones. Consider  the deterministic displacement model $\{n+\varepsilon\cdot \mathrm{frac}(\alpha n): n\in\mathbb Z\}$ for $\alpha\notin \mathbb Q$, see \cite[Ex.~9.8]{BG2} or \cite[Sec.~6]{dB87}, where $\mathrm{frac}(x)$ denotes the fractional part of $x$. This model is not strongly almost periodic, but it is a model set \cite[Ex.~9.8]{BG2}.
It is tempting to ask whether there is some different type of almost periodicity to cover such examples. A candidate is almost automorphicity, see e.g.~\cite{V65} and \cite[Thm.~1]{KR}. More generally, the above problem might be analyzed using functions with prescribed continuity properties on the Bohr compactification \cite{Reich70}. The latter problem may of course also be analyzed from a dynamical perspective. Consider the hull of a point set, i.e., its translation orbit closure in a Hausdorff-type metric. The hull of any SAP point set is a group, see Lemma~\ref{lem:sap}.   One may now ask for which deformation functions  the hull is sufficiently close to a group in order to retain pure point diffractivity.

\smallskip

Here is the structure of the article. As a motivating example, we discuss the diffraction of sine modulated integers. The succeeding general analysis uses mathematical diffraction theory and model sets in general locally compact abelian groups. Section~\ref{Sect} gives basic definitions and background on dynamical systems and diffraction theory. 
In Section~\ref{Strongly}, we treat dynamical and spectral properties of general SAP measures, from which we derive their diffraction. Although these results follow from the weakly almost periodic case, we decided to give specialized proofs for a self-contained presentation. Section~\ref{cpmodel} provides necessary background about cut-and-project schemes and model sets. Section~\ref{sec:lmd} is devoted to modulations and deformations of lattices, which extend our motivating example. The examples of Section~\ref{sec:lmd} are in fact deformed weighted model sets, which are shown to be  strongly almost periodic in Section~\ref{Smoothly}. 
Modulations of deformed weighted model sets are analyzed in
Section~\ref{sec:MDMS}. It is shown that the class of deformed weighted model sets is stable under modulation, and that in Euclidean space that class coincides with the class of modulated weighted model sets.
The final Section~\ref{The case of ideal crystals} treats modulations of ideal crystals, a class which comprises many examples from conventional crystallography. Hence our results in that section extend the analysis from the lattice case and specialize those for deformed model sets. In the Appendix, we discuss almost periodic functions taking values in a topological group.

\section{Sine modulated integers}\label{sec:sinemod}

\subsection{Overview}

Consider $f:\mathbb Z\to \mathbb R$ given by $f(\ell)=\varepsilon\cdot \sin(2\pi\alpha \ell)$ with $\varepsilon$ positive and $\alpha$ irrational. The point set $\Lambda=\{\ell+f(\ell): \ell\in\mathbb Z\}$ is called \textit{sine modulated integers}. It is a simple example of a modulated crystal \cite[Sec.~3.3]{JJ14}.%
\footnote{If $\alpha$ is rational, then the modulation function $f$ is periodic, and $\Lambda$ is a so-called ideal crystal. Such structures do not differ much from a lattice and will be discussed in Section~ \ref{The case of ideal crystals}, see especially Remark~\ref{rem:comm}.}
Since $\alpha$ is irrational, the point set $\{f(\ell): \ell\in\mathbb Z\}$ is dense in $[-\varepsilon,\varepsilon]$. This holds as the point set $\{\alpha \ell \mod 1: \ell\in\mathbb Z\}$ is dense in $[0,1]$, a well-known fact that may be regarded as a consequence of Weyl's equidistribution theorem, see e.g.~\cite[Ch.~1, Ex.~2.1]{KN74}. In fact $f$ is an almost periodic function, see the Appendix for background.

\medskip

For the distance $d_\ell$  between points arising from two consecutive integers $\ell$ and $\ell+1$ we compute
\begin{displaymath}
\begin{split}
d_\ell&= (\ell+1+f(\ell+1))-(\ell+f(\ell))=
1+\varepsilon\cdot \sin(2\pi\alpha(\ell+1))-\varepsilon\cdot \sin(2\pi\alpha \ell)\\
&= 1+2\varepsilon\cdot \cos(2\pi\alpha(\ell+1/2))\sin(\pi\alpha) \ .
\end{split}
\end{displaymath}
We thus have $\overline{\{d_\ell: \ell\in \mathbb Z\}}=[1-2\varepsilon \cdot |\sin(\pi\alpha)|, 1+2\varepsilon \cdot |\sin(\pi\alpha)|]$. In particular for $\varepsilon <1/2$ the set $\Lambda$ is uniformly discrete, and we may label the points $\ell+f(\ell)$ in ascending order by $\ell$. If $\varepsilon\ge 1/2$, then $\Lambda$ fails to be uniformly discrete. In either case $\Lambda$ has infinite local complexity, i.e., $\Lambda-\Lambda$ is not uniformly discrete.

\smallskip

 There is a topological dynamical system $\Omega$ with shift action naturally assigned to $\Lambda$, which is called its hull. It will be constructed in Section~\ref{sec:hull}.
 The hull contains all translates of $\Lambda$ and certain additional point sets which, loosely spoken, do not differ much from the translates. 
 The relation between the dynamical spectrum of $\Omega$ and the diffraction spectrum of $\Lambda$ allows us to compute its diffraction explicitly in Section~\ref{sec:diffsmi}. In fact any point set of the hull has the same diffraction.
This dynamical approach is complementary to the usual one of computing the Fourier--Bohr coefficients by a limiting procedure. It is reviewed in Section~\ref{difftheo} in a general setting.
 
\smallskip

There is a simple parametrization of the hull of sine modulated integers by a two-dimensional torus $\mathbb T$ with induced shift action. We call the corresponding factor map $\mu:\mathbb T\to \Omega$ the torus parametrization%
\footnote{In the context of model sets, the name was coined in \cite{BHP} for a suitable inverse of $\mu$, see also \cite{Ro96} and \cite{Martin2}. For a historical discussion and a recent dynamical approach see \cite{KR,KR17}.} 
of $\Omega$. Since the induced shift action on $\mathbb T$ is a dense rotation on the compact group $\mathbb T$, one can conclude that sine modulated integers are pure point diffractive \cite[Sec.~3]{BL2}. 
If $\varepsilon<1/2$, then the factor map is a homeomorphism, which implies that sine modulations are strongly almost periodic by Lemma~\ref{lem:sap}. In Section~\ref{sec:lmd} we will prove by a different method that general sine modulated integers are strongly almost periodic and hence pure point diffractive, if one properly takes into account multiplicities, compare Remark~\ref{rem:mult}.

\smallskip

Before we do the calculations, let us describe the diffraction of sine modulated integers for $\varepsilon<1/2$.
The Bragg peak positions lie dense in Fourier space and are explicitly given by $\mathbb Z[\alpha]=\{m+\alpha n:m,n\in \mathbb Z\}$. They have positive intensities
\begin{equation}\label{eq:amp}
|a_{m,n}|^2=\left|\int_0^1 \euler^{2\pi \imi (ns+(m+\alpha n)\,\varepsilon\cdot\sin(2\pi s))} {\rm d}s\right|^2\,.
\end{equation}
The diffraction of sine modulated integers is depicted in Figure~\ref{fig:diff} for $\varepsilon=1/20$ and $\alpha=1/\tau^4$, where $\tau$ denotes the golden mean. 

In general, for small modulation strength $\varepsilon>0$ one observes Bragg peaks of high intensity at the lattice positions, accompanied by Bragg peaks of small intensity.
If a Bragg peak of small intensity is close to some lattice Bragg peak, it is sometimes called a satellite. In particular, this is the case if the modulation length is of a different order of magnitude than the lattice
constant. 

\begin{figure}
\includegraphics[scale=0.45]{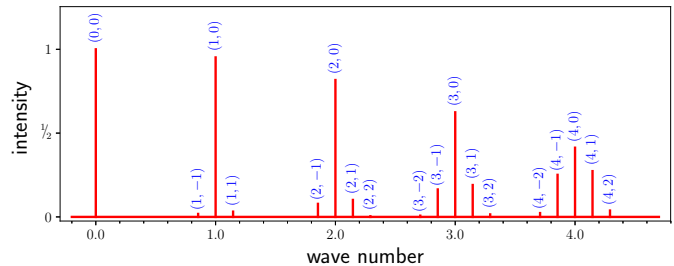}
\caption{Diffraction of sine modulated integers ($\varepsilon=1/20$, $\alpha=1/\tau^4$). Label $(m,n)$ corresponds to Bragg peak position $m+\alpha n$.}
\label{fig:diff}
\end{figure}

\subsection{The hull and its torus parametrization}\label{sec:hull}

Recall that the point set $\Lambda\subset \RR$ is uniformly discrete if and only if $\varepsilon<1/2$. In general $\Lambda$ is locally finite, i.e., discrete and closed. In fact $\Lambda$ is translation bounded, i.e., there exists finite constants $C$ and $R$ such that any intersection with a ball of radius $R$ contains at most $C$ points.

\smallskip

The hull of $\Lambda$ is the translation orbit closure of $\Lambda$ in a suitable topology, which we now describe.
Let us consider the collection $\mathcal D=\mathcal D_{C,R}$ of point sets that are translation bounded with constants $C,R$. We equip $\mathcal D$ with the \textit{local topology}, a Hausdorff type uniform topology that is frequently used when studying point set dynamical systems \cite[Sec.~4]{BL}. It is generated by the entourages
\[
U_{R, \delta} : = \{(\Lambda, \Lambda') \in \mathcal{D} \times \mathcal{D} : B_R \cap \Lambda \subset B_{\delta} + \Lambda'
\ \mbox{and} \ B_R \cap \Lambda' \subset B_{\delta} + \Lambda \}\,,
\]
where $R, \delta$ vary over the positive real numbers. Here $B_r$ denotes the centered closed ball of radius $r$.
In particular, for every $\Lambda\in\mathcal D$, the collection of all
$U_{R, \delta}(\Lambda) = \{\Lambda' \in \mathcal{D} : (\Lambda, \Lambda') \in U_{R, \delta}\}$,
where $R,\delta$ vary over the positive reals, is a neighborhood system of $\Lambda$. It turns out that $\mathcal D$ is a compact metrizable Hausdorff space \cite[Thm.~3]{BL}. In particular, we can use sequences for convergence arguments. Two point sets are close in the local topology, if they coincide on large centered balls up to small local shifts.

\smallskip

Consider the hull $\Omega = \overline{\{x + \Lambda : x \in \RR\}}\subset \mathcal D$
of $\Lambda$ in the local topology, with point set translation as natural $\RR$-action. Since $\alpha$ is irrational, it is not hard to see that $\Omega=\{\Lambda_{r,s}: (r,s)\in \RR^2\}$, where $\Lambda_{r,s}=\{r+\ell+\varphi(\alpha\ell+s):\ell\in \ZZ\}$ and $\varphi(x)=\varepsilon\cdot \sin(2\pi x)$. Due to the symmetries $\Lambda_{r,s+1}=\Lambda_{r+1,s+\alpha}=\Lambda_{r,s}$, the hull is parametrized by a two-dimensional torus. We will describe in Section~\ref{Strongly} how such a torus arises for strongly almost periodic structures in a natural way.

\smallskip

There is another natural construction of the torus, which will be central in Section~\ref{sec:lmd} below:
Consider the function space $H=\overline{\{f(\cdot+ \ell): \ell \in \ZZ\}}$, where the closure is taken using the topology of uniform convergence. It carries a natural group structure $\dot{+}$ defined via $f(\cdot + \ell) \dot{+} f(\cdot + \ell')=f(\cdot + (\ell+\ell'))$.
The LCA group $H$ is a compactification of $\ZZ$, since $f$ is almost periodic.
In fact $H=\{\varphi(s+\alpha \cdot \, ): s\in [0,1)\}$. Thus $H$ is homeomorphic to a one-dimensional torus. 
Now $\LL=\{(\ell, f(\cdot+\ell)): \ell \in \ZZ\}\subset \RR\times H$
is a group and a Delone set, i.e., $\LL$ is lattice in $\RR\times H$.
Consider $\TT=(\RR \times H) / \LL$ with natural $\RR$-action
$s+[r,g]=[s+r, g]$ and define $\mu: \TT \to \Omega$ by
\begin{displaymath}
\mu([r,g]) = r + \{\ell+g(\ell):\ell \in \ZZ\} \ .
\end{displaymath}

\smallskip

\begin{lemma}
The map $\mu: \TT\to \Omega$ is a factor map, i.e., it is continuous onto and commutes with the natural $\RR$-actions on $\TT$ and on $\Omega$.  If $\varepsilon<1/2$, then $\mu$ is a homeomorphism.
\end{lemma}

\begin{proof}

We show that $\mu$ is well-defined, onto and continuous. We then show that  $\mu$ is an $\RR$-map. We finally show that $\mu$ is one-to-one if $\varepsilon<1/2$, which implies that $\mu$ is a homeomorphism in that case since $\Omega$ is compact.

\smallskip

\noindent \textit{$\mu$ is well-defined:} Take arbitrary $(r,g)\in \mathbb R\times H$ and write $g=\lim_{k\to\infty} f(\cdot-\ell_k)$ for some integer sequence $(\ell_k)_k$.  Then $\mu([r,g])\in \Omega$ since
\begin{displaymath}
\mu([r,g]) = \{r+\ell+\lim_{k\to\infty} f(\ell-\ell_k):\ell \in \ZZ\}=\lim_{k\to\infty}\left(  r+\ell_k+ \{\ell+f(\ell):\ell \in \ZZ\}\right).
\end{displaymath}
The latter limit indeed exists in $\Omega$, as any accumulation point of $\left( r+\ell_k+ \Lambda\right)_k$ in the compact space $\Omega$ coincides with $\mu([r,g])$.
The definition is independent of the choice of representative: Let $(\ell, f(\cdot + \ell))\in \LL$ and note
\begin{eqnarray*}
\mu([r+\ell, g(\cdot) \dot{+} f(\cdot + \ell)]) & = & r+\ell+\{\ell'+g(\ell') \dot{+} f(\ell'+\ell) : \ell' \in \ZZ\} \\
& = & r+\ell+\{\ell'+ \lim_{k \to \infty} f(\ell'-\ell_k) \dot{+} f(\ell'+\ell) : \ell' \in \ZZ\} \\
& = & r+\ell+\{\ell'+ \lim_{k \to \infty} f(\ell'-\ell_k + \ell) : \ell' \in \ZZ\} \\
& = & r+\ell+\{\ell' - \ell + \lim_{k \to \infty} f(\ell'-\ell_k) : \ell' \in \ZZ\} \\
& = & r+\{\ell' + g(\ell') : \ell' \in \ZZ\} = \mu([r, g])\,.
\end{eqnarray*}

\noindent \textit{$\mu$ is onto:}
Consider any $\Lambda' \in \Omega$. For every $k \in \NN$ choose $r_k \in \RR$ such that
\[
r_k + \{\ell+ f(\ell): \ell \in \ZZ\} \in U_{k,\frac{1}{k}}(\Lambda')\,.
\]
By passing to a suitable subsequence of $(r_k)_k$, we can find integers $(\ell_{k})_k$ such that $\lim_{k \to \infty} (r_{k} - \ell_{k}) = r$ for some $r \in [0,1)$. Note that
\[
r_{k} + \{\ell + f(\ell) : \ell \in \ZZ \} = r_{k} - \ell_{k} + \{\ell' + f(\ell' - \ell_{k}) : \ell' \in \ZZ \}\,.
\]
Using compactness of $H$, we may assume that $f(\cdot -  \ell_{k})$ converges,
by passing to a suitable subsequence. Now define $g = \lim_{k \to \infty} f(\cdot - \ell_{k})$. We then have
$\Lambda' = r + \{\ell+ g(\ell): \ell \in \ZZ\}$.
Thus $[r, g]\in \TT$ satisfies $\mu([r, g]) = \Lambda'$.

\smallskip

\smallskip

\noindent \textit{$\mu$ is continuous:}  Consider an arbitrary member $\mu([r,g])$ in the hull. Fix arbitrary $R>0$ and $\varepsilon>0$ and consider $r'\in B_{\varepsilon/2}(r)\subset \RR$ and $g'\in B_{\varepsilon/2}(g)\subset H$. By the triangle inequality, we then have
\begin{displaymath}
\mu([r',g'])\subset B_\varepsilon+\mu([r,g])\,, \qquad \mu([r,g])\subset B_\varepsilon+\mu([r',g'])\,.
\end{displaymath}
In particular we have $\mu([r',g'])\in U_{R, \varepsilon}(\mu([r, g)])$. As $R>0$ and $\varepsilon>0$ were arbitrary, this shows that $\mu$ is continuous.

\smallskip

\noindent \textit{$\mu$ is an $\RR$-map:} Take any $s\in\mathbb R$ and note
\[
\mu(s + [r, g]) = \mu([s + r, g]) = s + r + \{\ell + g(\ell): \ell \in \ZZ \}= s+ \mu([r, g])\,.
\]
\smallskip

\noindent \textit{$\mu$ is one-to-one for $\varepsilon<1/2$:}
Consider $\mu([r,g])=\mu([s,h])$ and assume $r\in [0,1)$ and $s=0$  without loss of generality.  Write $\mu([r,g])=\{p_\ell:\ell\in\ZZ\}$ and  $\mu([0,h])=\{q_\ell:\ell\in\ZZ\}$, where  $p_\ell=r+\ell+g(\ell)$ and $q_\ell=\ell+h(\ell)$. Note that $p_{\ell_1}<p_{\ell_2}$ and  $q_{\ell_1}<q_{\ell_2}$ whenever $\ell_1<\ell_2$ since  $\overline{g(\ZZ)} = \overline{h(\ZZ)}\subset (-1/2,1/2)$ due to irrationality of $\alpha$, and that we have
\begin{displaymath}
p_\ell \in (\ell-1/2, \ell+3/2) \ , \qquad q_\ell \in (\ell-1/2, \ell+1/2)
\end{displaymath}
for every $\ell\in \ZZ$. In particular either $p_0=q_0$ or $p_0=q_1$.
Assuming $p_0=q_1$, we infer $p_\ell=q_{\ell+1}$ for all $\ell\in \ZZ$. Thus $r+\overline{g(\ZZ)}=1+\overline{h(\ZZ)}$, which is impossible as $r\in [0,1)$.
We thus have $p_0=q_0$, which leads to $p_\ell=q_\ell$ for all $\ell\in \ZZ$. Thus $r+\overline{g(\ZZ)}=\overline{h(\ZZ)}$, which implies $r=0$ and $g=h$. Hence $\mu$ is one-to-one. \qed
\end{proof}

\subsection{Diffraction of sine modulated integers}\label{sec:diffsmi}

We compute the diffraction measure of sine modulated integers via the dynamical systems approach.
The torus parametrization $\mu:\TT\to \Omega$ can be used to embed $L^2(\Omega)$ isometrically into $L^2(\TT)$, compare \cite[Thm.~1]{BL2}.  (For $\varepsilon<1/2$, the torus parametrization is a homeomorphism, and both spaces can be identified.) In particular, since $(\TT, \RR)$ has pure point dynamical spectrum, $(\Omega, \RR)$ has pure point dynamical spectrum too, compare \cite[Prop.~1]{BL2}. The latter property is equivalent to $(\Omega, \RR)$ having pure point diffraction \cite[Thm.~7]{BL}.

\medskip

As $(\TT, \RR)$ is a dense rotation on a compact group, the space $L^2(\TT)$ admits an orthonormal basis of eigenfunctions $e_\xi$ with eigenvalue $\xi\in \widehat \RR \cong \RR$, compare \cite{Wal}. Every such $\xi$ is a potential Bragg peak position in the diffraction of $\Lambda$. For explicit expressions, let us denote by $\Gamma$ the invariance lattice of $\Omega=\{\Lambda_{r,s}: (r,s)\in \RR^2\}$ and its dual lattice by $\Gamma_0$. We have
\begin{displaymath}
\Gamma=\left\{ {m \choose n+\alpha m}: m,n\in \ZZ \right\} \ ,\qquad
 \Gamma_0=\left\{ {-m-\alpha n \choose n}: m,n\in \ZZ \right\}  \ .
\end{displaymath}
Eigenfunctions of the translation operator on $L^2(\TT)$, where
$\TT=\RR^2/\Gamma$, arise from lattice invariant plane waves $x\mapsto \chi_k(x)=\euler^{2\pi \imi k\cdot x}$, which restricts $k$ to elements in $\Gamma_0$.
Given $k=(-m-\alpha n,n)^T$, we will write $e_\xi$ instead of $\chi_k$ for $\xi=m+\alpha n$, which will cause no confusion as $k$ can be reconstructed from $\xi$. A simple calculation shows that $e_\xi\in L^2(\TT)$ is indeed an eigenfunction of the translation operator  with eigenvalue $\xi$, and $\{e_\xi : \xi\in \ZZ[\alpha]\}$ is total and orthonormal in $L^2(\TT)$. Using the isometric embedding, we can apply the diffraction formula Eqn.~\eqref{eq:FBC1} from Section~\ref{difftheo}. The diffraction intensity at $\xi$ is given by $|a_\xi|^2$, where
\begin{displaymath}
a_\xi=\int_\TT \overline{e_\xi(t)} f_{\xi\cdot\psi}(\mu(t)) {\rm d}t \ .
\end{displaymath}
Here $\psi\in C_c(\RR)$ is any function satisfying $\int_\RR \psi(x) {\rm d}x=1$, and we have used $f_\psi(\Lambda)=\sum_{p\in\Lambda} \psi(p)$. If $\mu$ is a homeomorphism, then $a_\xi\ne0$ for all $\xi\in \ZZ[\alpha]$.

\smallskip

Let us now assume $\varepsilon<1/2$ for simplicity. In order to compute $a_\xi$ for $\xi=m+\alpha n$, we write
\begin{displaymath}
\begin{split}
a_\xi&=\int_0^1\int_0^1 \overline{e_\xi((r,s))} f_{\xi\cdot\psi}(\Lambda_{r,s}) {\rm d}(r,s)\\
&= \sum_{\ell\in \ZZ} \int_0^1 \int_0^1 \xi(-r)\euler^{2\pi \imi n s} (\xi\cdot\psi)(r+\ell+\varphi(\alpha\ell+s)){\rm d}(r,s) \ .
\end{split}
\end{displaymath}
As $\psi$ has compact support, only finitely many terms are nonzero in the above sum. Suppose now that $\psi$ is sharply concentrated about $0$. Then, up to an arbitrarily small error, the above expression simplifies to
\begin{displaymath}
\begin{split}
a_\xi&= \sum_{\ell\in \{-1,0\}} \int_0^1 \int_0^1 \xi(\ell+\varphi(\alpha\ell+s))\euler^{2\pi \imi n s} \psi(r+\ell+\varphi(\alpha\ell+s)){\rm d}(r,s) \ .
\end{split}
\end{displaymath}
Consider the term corresponding to $\ell=0$ in the above sum. We can have $r+\varphi(s)=0$ for some $r\in (0,1)$ only if $s\in (1/2,1)$. Thus, up to an arbitrarily small error, that term equals
\begin{displaymath}
\begin{split}
 \int_{1/2}^1 &\int_0^1 \xi(\varphi(s))\euler^{2\pi \imi n s} \psi(r+\varphi(s)){\rm d}(r,s)=
\int_{1/2}^1 \xi(\varphi(s))\euler^{2\pi \imi ns} {\rm d}s\\
=& \int_{1/2}^1 \euler^{2\pi \imi (n s+(m+\alpha n) \varphi(s))} {\rm d}s \ .
\end{split}
\end{displaymath}
Next, consider the term corresponding to $\ell=-1$ in the above sum. We can have $r-1+\varphi(-\alpha+s)=0$ for some $r\in (0,1)$ only if $s\in (\alpha, \alpha+1/2)$. As we might choose a fundamental domain of $\Gamma$ arbitrarily for integration, that term equals
\begin{displaymath}
\begin{split}
 \int_{\alpha}^{\alpha+1/2} & \int_0^1 \xi(-1+\varphi(-\alpha+s))\euler^{2\pi \imi n s} \psi(r-1+\varphi(-\alpha+s)){\rm d}(r,s)\\
=&
\int_{\alpha}^{\alpha+1/2} \xi(-1+\varphi(-\alpha+s))\euler^{2\pi \imi ns} {\rm d}s=
\int_{0}^{1/2} \xi(-1+\varphi(s))\euler^{2\pi \imi n(s+\alpha)} {\rm d}s\\
=&\int_{0}^{1/2} \euler^{2\pi \imi (n s+(m+\alpha n) \varphi(s))} {\rm d}s \ .
\end{split}
\end{displaymath}
Combining the latter two results, we get $|a_\xi|^2=|a_{m,n}|^2$, compare Eqn.~\eqref{eq:amp}.

\section{Setting and notation}\label{Sect}

Throughout $G, H$ will denote locally compact abelian groups (LCAG) \cite{Rei2,Fol,DE}. The group operation on a LCAG will be written additively as $+$ or $\dotplus$ if necessary to avoid
misunderstandings. Given an LCA group $G$, we choose a Haar measure on $G$ and denote it by $m_G$ or by
${\rm d} t$. The dual group of $G$ is denoted by $\Ghat$, and the
pairing between a character $\chi \in \Ghat$ and $t\in G$ is
written as $\chi(t)$. As usual the Fourier transform $\widehat{f}$
of an integrable function $f$ is defined by $\widehat{f} (\chi) =
\int_G \overline{\chi(t)} f(t) \, {\rm d}t$. We always choose the Haar measure on $\widehat G$ such that the Plancherel theorem \cite[Thm.~3.4.8]{DE} holds.

\subsection{Complex Radon measures}

Whenever $X$ is a topological space
the space of continuous functions on $X$ is denoted by $C(X)$, the subspace of continuous functions with compact
support by $C_c (X)$, the space of continuous bounded functions by $C_b (X)$, and the subspace of uniformly continuous and bounded functions by $C_\mathsf{u}(X)$. The latter two spaces are complete normed spaces when equipped with the supremum norm $\|\cdot\|_\infty$. We will often deal with locally compact $\sigma$-compact spaces.

\smallskip

A topological space $X$ carries the Borel $\sigma$-algebra generated
by all closed subsets of $X$.  By the Riesz-Markov representation
theorem, the set $\mathcal{M} (X)$ of all complex Radon
measures on a locally compact space $X$ can then be identified with the dual space $C_c
(X)^\ast$ of complex valued, linear functionals on $C_c(X)$ which are
continuous with respect to a suitable topology, see \cite[Ch.\
6.5]{Ped} for details. For this reason we usually write
 $\int_X \varphi \dd\mu =\mu(\varphi)$ for $\varphi\in C_c(X)$. 
The space $\mathcal{M} (X)$ then  carries the
vague topology, i.e., the weakest topology that makes all functionals
$\mu\mapsto \mu(\varphi)$, $\varphi\in C_c (X)$, continuous.
If $X$ is metrizable, then $\mathcal M(X)$ is metrizable as well. The convolution $\mu \ast \nu$ of two finite
measures $\mu$ and $\nu$ on $G$ is defined to be the measure $(\mu\ast \nu)
(\varphi) :=\int \int \varphi(s+t) \dd\mu(s) \dd\nu(t)$.

\smallskip

By \cite[Thm.~6.5.6]{Ped}, for each measure 
 $\mu \in \mathcal{M} (X)$ there exists a positive measure $|\mu|$, called the total variation of $\mu$ such that
for all $\varphi \in C_c(X)$ with $\varphi \geq 0$ we have
\[
|\mu|(\varphi) = \sup \{ \left| \mu(\psi) \right| : \psi \in C_c(X), |\psi| \leq \varphi \} \,.
\]

Now let $G$ be a LCAG. A measure $\mu\in \MM $ is called {\em translation bounded\/} if
there exist some  $C>0$ and an open nonempty relatively compact  set $V$ in $G$  so
that
\begin{equation}\label{mcv} |\mu| (t+V) \leq C
\end{equation}
for every $t\in G$.
The set of all translation bounded measures
satisfying \eqref{mcv} is denoted by $\MCV$. It carries the vague topology inherited from $\MM$ and is compact in that topology. The set of all translation bounded measures is denoted by $\MTB$.

\subsection{Measure dynamical systems}\label{sec:mds}
Whenever the LCAG  $G$ acts on the compact Hausdorff  space $\varOmega$  by a continuous action
\begin{equation*}
   \alpha \! : \; G\times \varOmega \; \longrightarrow \; \varOmega
   \, , \quad (t,\omega) \, \mapsto \, \alpha^{}_{t} (\omega) \, ,
\end{equation*}
where $G\times \varOmega$ carries the product topology, the pair
$\Oomega$ is called a {\em topological dynamical system\/} over $G$.
An $\alpha$-invariant probability measure on $\varOmega$ is then
called {\em ergodic\/} if every measurable invariant subset of
$\varOmega$ has either measure zero or measure one. The dynamical
system $\Oomega$ is called {\em uniquely ergodic\/} if there exists a
unique $\alpha$-invariant probability measure on $\varOmega$, which
then is ergodic by standard theory. $\Oomega$ is called {\em
minimal\/} if, for all $\omega\in\varOmega$, the $G$-orbit
$\{\alpha^{}_t\ts \omega : t \in G\}$ is dense in $\varOmega$.

Given an $\alpha$-invariant probability measure $m$ on $\varOmega$,
we can form the Hilbert space $\LO$ of square integrable measurable
functions on $\varOmega$. This space is equipped with the inner
product
\begin{equation*}
    \langle f, g\rangle \; = \;
    \langle f, g\rangle^{}_\varOmega \; := \;
    \int_\varOmega \overline{f(\omega)}\, g(\omega) \dd m(\omega) \ .
\end{equation*}
The action $\alpha$ gives rise to a unitary representation $T :=
T^\varOmega := T^{(\varOmega,\alpha,m)}$ of $G$ on $\LO$ by
\begin{equation*}
  T_t \! : \; \LO \; \longrightarrow \; \LO \, ,
  \quad (T_t f) (\omega) \; := \;
  f(\alpha^{}_{-t}\ts \omega) \, ,
\end{equation*}
for every $f\in \LO$ and arbitrary $t\in G$.
An $f\in \LO$ is called an {\em eigenfunction\/} of $T$ with
{\em eigenvalue\/} $\xi\in \Ghat$ if  for every
$t\in G$ we have $T_t f = \xi(t) f$.  An eigenfunction (to $\chi$, say) is called {\em continuous\/}
if it has a continuous representative $f$  with
$f(\alpha^{}_{-t} \ts \omega) = \xi(t)\ts f(\omega)$, for all
$\omega\in\varOmega$ and $t\in G$. The representation $T$ is said to have
{\em pure point spectrum\/} if the set of eigenfunctions is total in $\LO$.
One then also says that the dynamical system $\Oomega$ has
{\em pure point dynamical spectrum}.

  Let two topological dynamical systems\/ $\Oomega$ and\/ $\Ttheta$
  under the action of $G$ be given.  Then $\Ttheta$ is called a\/
  {\em (topological) factor} of $\Oomega$, with factor map\/ $\varPhi$, if\/
  $\varPhi \! : \varOmega \longrightarrow\varTheta $ is a continuous
  surjection with $\varPhi (\alpha^{}_t (\omega)) = \beta^{}_t
  (\varPhi (\omega))$ for all\/ $\omega\in \varOmega$ and\/ $t\in G$.

\smallskip

We will be concerned with dynamical systems built from measures.  These
systems will be discussed next. They have been introduced in \cite{BL,BL2}, to
which we refer for further details and proofs of the subsequent discussion.
There is an obvious action of $G$ on $\MTB$, again denoted
by $\alpha$, given by
\begin{equation*}
   \alpha \! : \; G\times  \MTB \; \longrightarrow \; \MTB
   \, , \quad (t,\nu) \, \mapsto \, \alpha^{}_t\ts \nu
   \quad \mbox{with} \quad (\alpha^{}_t \ts \nu)(\varphi) \, := \,
   \nu(\delta_{-t} \ast \varphi)
\end{equation*}
for $\varphi \in C_c (G)$.  Here, $\delta_t$ denotes the unit point
mass at $t \in G$ and the convolution $\omega \ast \varphi$ between
$\varphi\in C_c (G)$ and $\omega \in \MTB$ is defined by
\[ 
(\omega \ast \varphi) (s) := \int \varphi (s - u) \dd \omega(u) \ .
\]
It is not hard to see that $\alpha$ is continuous when restricted to a
compact subset of $\MTB$.

\begin{definition}
  $\Oomega$ is called a dynamical system on the translation bounded
  measures on\/ $G$  {\rm (TMDS)} if
  $\varOmega$ is a compact  $\alpha$-invariant subset of\/ $\MCV$ for some open relatively compact $V$ and $C>0$.
 \end{definition}

Every translation bounded measure $\nu$ gives rise to a TMDS
$(\Omega(\nu),\alpha)$ \cite{BL}, where
\begin{displaymath}
\Omega(\nu):=\overline{\{\alpha_t\nu : t\in G\}}\,.
\end{displaymath}
More precisely, if $\nu\in \MCV$, then $\varOmega(\nu) \subset \MCV$. We call $\Omega(\nu)$ the hull of $\nu$.

\subsection{Mathematical diffraction theory}\label{difftheo}

We review elements of mathematical diffraction theory following \cite{BL,BL2,Len}.
Fix an LCA group $G$ and assume that $G$ is $\sigma$-compact. Then $G$ admits a \textit{van Hove sequence} \cite{Martin2}, i.e., a sequence $(B_n)_{n\in\mathbb N}$ of compact sets in $G$ such that for every  compact $K\subset G$ we have
$$
\lim_{n\to \infty} \frac{m_G (\partial^K B_n)}{m_G (B_n)} =0 \ .
$$
Here for arbitrary compact $A,K\subset G$ we use the notion of van Hove boundary
\begin{displaymath}
\partial^K A:= (( K + A )\setminus A^\circ ) \cup (( - K +\overline{G\setminus A})\cap A ) \ ,
\end{displaymath}
where the bar denotes the closure of a set and the circle denotes the
interior. In $G=\mathbb R^d$, the closed $n$-balls constitute a van Hove sequence. Fix $\omega\in \mathcal M^\infty(G)$, such as the Dirac comb of a point set in $G$, and consider for arbitrary  $\xi\in \widehat G$ the average
\begin{equation}
a_\xi = \lim_{n\to \infty} \frac{1}{m_G(B_n)} \int_{B_n} \overline{\xi (t)} \, {\rm d}\omega(t) \ .
\end{equation}
 In this article, the above limit will always exist and will always be independent of the choice of the van Hove sequence. In crystallography, the number $a_\xi$ is called the \textit{scattering amplitude} or the \textit{Fourier--Bohr coefficient}. In a kinematic diffraction experiment \cite{Cowley}, the  intensity of diffraction at $\xi$ is given by $|a_\xi|^2$.

\smallskip

Hof \cite{Hof} suggested the following mathematical framework for diffraction. Given $\omega\in \mathcal M^\infty(G)$, one first computes the \textit{autocorrelation measure}
\begin{displaymath}
\gamma_\omega = \lim_{n\to \infty} \frac{1}{m_G(B_n)} \omega_{B_n} \ast
\widetilde{\omega_{B_n}} 
\end{displaymath}
with respect to the given van Hove sequence $(B_n)_{n\in \mathbb N}$.
Here the limit is taken in the vague topology,  $\omega_{B_n}$ denotes the
 restriction of $\omega$ to $B_n$, and the reflected measure
$\widetilde{\mu}$ is, for $\mu\in \mathcal M^\infty(G)$,  defined by $\widetilde{\mu}(\varphi):=\overline{ \mu(
  \overline{ \varphi^\dag} })$, where $\varphi^\dag(x)=\varphi(-x)$.  In this article, the above limit will always exist and will always be  independent of the choice of the van Hove sequence. More generally, this is the case if the hull $(\Omega(\omega), \alpha)$ is uniquely ergodic.%
\footnote{In general, a limit will always exist on some subsequence of the given van Hove sequence.}
As $\gamma_\omega$ is a positive definite measure by construction, its  Fourier transform $\widehat{\gamma_\omega}$ exists as a measure. It is called the
\textit{diffraction measure} of $\omega$  (see \cite{Fol,BF,ARMA1,GdeL,MoSt} for definition and background on Fourier transforms of complex Radon measures). Its point part is  explicitly given by
\begin{equation}\label{eq:FBC2}
\widehat{\gamma_\omega} (\{\xi\}) = |a_\xi|^2=\left|\lim_{n\to \infty} \frac{1}{m_G(B_n)} \int_{B_n} \overline{\xi (t)} \, {\rm d}\omega(t) \right|^2 \ ,
\end{equation}
see e.g.~\cite[Thm.~3.4]{Hof}.
We say that $\omega\in \mathcal M^\infty(G)$ is \textit{pure point diffractive} (relative to the given van Hove sequence) if the autocorrelation $\gamma_\omega$ exists and its Fourier transform $\widehat{\gamma_\omega}$ is a pure point measure. 

\medskip

We now turn to a dynamical description of diffraction \cite{BL,Len}. Let $G$ be an arbitrary LCAG. Let $\Oomega$ be a TMDS, equipped with an $\alpha$-invariant measure $m$. Let $\psi \in C_c (G)$ with $\int \psi(t) \dd t =1$ be given. Then,
$\gamma_m : C_c (G) \longrightarrow \CC$ defined by
\[ \gamma_m (\varphi) := \int_\varOmega \int_G \int_G \varphi ( s + t) \psi
(t) \dd \omega(s) \dd \widetilde{\omega}(t) \dd m (\omega)\,,
\]
is a positive definite measure which does not depend on the choice of $\psi$. 
 The measure $\gamma_m$ is called the autocorrelation measure of $(\Omega, \alpha, m)$.  If $e_\xi\in L^2(\Omega, m)$ is a normalized eigenfunction of $T$ with eigenvalue $\xi\in \widehat G$, we then have
\begin{equation}\label{eq:FBC1}
\widehat{\gamma_m}(\{\xi\}) = \left|\int_{\Omega} \overline{e_\xi(\omega)} \cdot \omega(\xi\cdot\psi) \, {\rm d}m(\omega)\right|^2 \, ,
\end{equation}
where $\psi\in C_c(G)$ such that $\int_G \psi(t) {\rm d}t=1$ is arbitrary. The measure $\widehat{\gamma_m}$ is called the diffraction measure of $(\Omega, \alpha, m)$.

\smallskip

Assume now that $G$ is $\sigma$-compact and fix $\nu\in \mathcal M^\infty(G)$.  If $(\Omega(\nu),\alpha)$ is uniquely ergodic, we have $\widehat{\gamma_m}=\widehat{\gamma_\omega}$ for every $\omega\in \Omega(\nu)$, compare \cite[Thm.~3, Thm.~5]{Len}.%
\footnote{If $(\Omega(\nu),\alpha)$ fails to be uniquely ergodic but $m$ is ergodic, then the latter equality only holds almost surely with respect to $m$, by the Birkhoff ergodic theorem. }
In particular this means that the diffraction amplitudes can be computed in two different ways using Eqn.~\eqref{eq:FBC2} or Eqn.~\eqref{eq:FBC1}. 
In this article, we will always deal with uniquely ergodic $(\Omega(\nu),\alpha)$. We will always compute the diffraction measure using the dynamical approach. In particular, we do not need to impose $\sigma$-compactness of $G$.

\section{Strongly almost periodic measures}\label{Strongly}

In Section~\ref{sec:sinemod} we discussed sine modulated integers. We showed that its hull is a compact abelian group, an insight which enabled us to infer the diffraction of sine modulated integers. The present section is devoted to a discussion of the general phenomenon, where we use measures instead of point sets. Note that diffraction has recently been studied for more general weakly almost periodic measures \cite{LenStr16}. For the ease of the reader we include proofs adapted to our simpler setting.

\smallskip

Consider $\nu\in \mathcal M^\infty(G)$, such as the Dirac comb of a Delone set in $G$. The following lemma characterises the situation that its hull $\Omega(\nu)$ is a group compactification of $G$ in terms of strong almost periodicity of $\nu$. This is certainly well known, see e.g.~\cite[Prop.~3.6]{LenStr16} and compare  \cite[Thm.~3.3]{KL13} and \cite[Thm.~4.2]{KL13} for Delone sets. We include a proof based on \cite{LR} for the convenience of the reader.

\smallskip

Recall that a translation bounded measure $\nu\in\mathcal{M}^\infty(G)$ is \textit{strongly almost periodic} if the function $\nu \ast \varphi\in C_\mathsf{u}(G)$ is almost periodic for every $\varphi \in C_c (G)$. As usual,  $f \in C_\mathsf{u}(G)$ is almost periodic if $\{\delta_t \ast f : t\in G\}$ is relatively compact  in $C_\mathsf{u} (G)$. See Proposition~\ref{prop:BochBohr} in Appendix~\ref{app:gvap} or \cite{C68,K04}  for background.

\begin{lemma}\label{lem:sap}
Let $\nu \in \MTB$ be given. Then, the following assertions are equivalent:
\begin{itemize}
\item[(i)] The hull $\Omega(\nu)$ of $\nu$ is a compact abelian group, with the addition $\dot{+}$ satisfying $\alpha_s \nu \dot{+} \alpha_t \nu = \alpha_{s+t} \nu$ for any $s, t \in G$.
\item[(ii)] The measure $\nu$ is strongly almost periodic.
\item[(iii)] Every $\omega\in\Omega(\nu)$ is strongly almost periodic.
\end{itemize}
In this case, the map  $\jmath: G \to \Omega(\nu)$ defined by $\jmath(t) = \alpha_{t} \nu$ is a continuous group homomorphism with dense range, and
$$\alpha_{t} \omega = \jmath(t) \dot{+} \omega$$
holds for any $t\in G$ and any $\omega \in \Omega(\nu)$.
\end{lemma}

\begin{proof}
The equivalence between (i) and (ii) can be inferred from \cite[Lemma\,4.2]{LR}.
For the equivalence between (ii) and (iii), consider any $\varphi\in C_c(G)$ and any $\varepsilon>0$. Then all functions $\omega *\varphi\in C_\mathsf{u}(G)$, where $\omega$ ranges over $\varOmega(\nu)$, have the same $\varepsilon$-almost periods.

 Finally, the map is continuous with dense range.  It is a group homomorphism when $\Omega (\nu)$ carries the group structure defined in (i). It remains to compute the action $\alpha$ in terms of $j$:
If $\omega = \alpha_s \nu$, then for any $\varphi \in C_c(G)$,
\begin{eqnarray*}
(\alpha_t\omega)(\varphi) &=& \omega(\delta_{-t} \ast \varphi) = (\alpha_s \nu) (\delta_{-t} \ast \varphi) = \nu(\delta_{-s} \ast \delta_{-t} \ast \varphi) \\
 &=& \nu(\delta_{-s-t} \ast \varphi) = (\alpha_{s+t} \nu)(\varphi) = (\alpha_t \nu \dot{+} \alpha_s \nu)(\varphi) \\
 &=& (\jmath(t) \dot{+} \omega)(\varphi)\,.
\end{eqnarray*}
Let $\omega = \lim_{\iota} \alpha_{s_\iota}\nu$, i.e., $(\alpha_{s_\iota}\nu)_\iota$ is a net converging to $\omega$. Then
\begin{eqnarray*}
(\alpha_t\omega)(\varphi) &=& \omega(\delta_{-t} \ast \varphi) = \lim_{\iota} (\alpha_{s_\iota} \nu)(\delta_{-t} \ast \varphi)
= \lim_{\iota} \nu(\delta_{-s_\iota} \ast \delta_{-t} \ast \varphi) \\
&=& \lim_{\iota} (\alpha_{t+ s_\iota} \nu)(\varphi) = \lim_{\iota} (\alpha_{t} \nu \dot{+} \alpha_{s_\iota} \nu)(\varphi)
= (\alpha_{t} \nu \dot{+} \lim_{\iota} \alpha_{s_\iota} \nu)(\varphi) \\
&=& (\jmath(t) \dot{+} \omega)(\varphi)\,.
\end{eqnarray*}
\qed
\end{proof}

We can now readily derive basic spectral properties of dynamical systems associated to strongly almost periodic measures.  For notation and basic results of mathematical diffraction theory, recall Section~\ref{Sect}.

\begin{theorem} \label{stronglyAlmostPeriodic-implies-purePointSpectrum}
Let $\nu \in \MTB$ be strongly almost periodic and $\jmath: G\longrightarrow \Omega (\nu)$ the canonical mapping described in Lemma~\ref{lem:sap}.  Then the following hold.
\begin{itemize}
\item[(i)] $(\Omega(\nu), \alpha)$ is uniquely ergodic and minimal.
\item[(ii)]
For the shift operator $T$ on $L^2(\Omega(\nu))$, any character $\lambda:\Omega(\nu) \to S^1$ is an eigenfunction of $T$ to the eigenvalue $\lambda \circ \jmath: G\to S^1$.  The set $\widehat{\Omega(\nu)}$ of continuous
characters  provides an orthonormal basis of $L^2(\Omega(\nu))$.  In particular, every character is continuous, and the measure dynamical system
$(\Omega(\nu), \alpha)$ has pure point dynamical spectrum.

\item[(iii)] The measure $\nu$ has a unique autocorrelation $\gamma$. Its Fourier transform, the diffraction measure  $\widehat{\gamma}\in \mathcal M^\infty(\widehat G)$, is pure point and given by
\begin{displaymath}
\widehat{\gamma} = \sum_{\lambda \in \widehat{\Omega(\nu)}} |a_{\lambda\circ \jmath }|^2\, \delta_{\lambda \circ \jmath} \, ,
\end{displaymath}
where the constants $a_{\lambda\circ \jmath}$ are given by
\begin{displaymath}
a_{\lambda\circ \jmath}=\int_{\Omega(\nu)} \overline{\lambda(\omega)} \, \omega(\lambda\circ \jmath \cdot \varphi) \, {\rm d} \omega \ ,
\end{displaymath}
and $\varphi\in C_c(G)$ is any function satisfying $\int_G \varphi(t) \, {\rm d}t =1$.
\end{itemize}
\end{theorem}

\begin{remark}
If $\nu$ is a Dirac comb, it is possible to interpret minimality geometrically in terms of the underlying point set. The underlying point set is almost repetitive \cite[Thm.~3.11]{FR14}.
\end{remark}

\begin{proof}
(i) Let $m$ be any $G$-invariant probability measure on the hull $\Omega(\nu)$. Let $f \in C(\Omega(\nu))$ be arbitrary. Then we have for every $t\in G$ the equalities
\begin{displaymath}
\int f(\jmath(t) \dot{+} \omega) \, {\rm d}m(\omega) = \int f(\alpha_{-t}\omega) \, {\rm d}m(\omega) = \int f(\omega) \, {\rm d}m(\omega)\,,
\end{displaymath}
as $m$ is $G$-invariant.
Since $\jmath(G)$ is dense in $\Omega(\nu)$ and $f$ is continuous on $\Omega(\nu)$, we have
\begin{displaymath}
\int f(\sigma \dot{+} \omega) \, {\rm d}m(\omega) = \int f(\omega) \, {\rm d}m(\omega)
\end{displaymath}
for every $\sigma \in \Omega(\nu)$.
Thus $m$ is a Haar measure. Since $m$ is a probability measure, $m$ is unique. This implies that $\Omega(\nu)$ is minimal, see e.g.~\cite[Thm.~6.17]{Wal}. Here is a direct argument: Assume $\omega=\lim_\iota \alpha_{t_\iota}\nu \in \Omega(\nu)$. Then $\Omega(\omega)\subset \Omega(\nu)$ by translation invariance and closedness of $\Omega(\nu)$. It is also readily seen that $\nu=\lim_\iota \alpha_{-t_\iota}\omega$, which implies $\Omega(\nu)\subset \Omega(\omega)$. Thus $\Omega(\omega)=\Omega(\nu)$ for all $\omega\in \Omega(\nu)$, which shows minimality.

\smallskip

\noindent (ii) Let $\lambda:\Omega(\nu)\to S^1$ be any character. Then
\begin{displaymath}
(T_t \lambda)(\omega) = \lambda(\alpha_{-t} \omega) = \lambda(\jmath(t) \dot{+} \omega) = \lambda(\jmath(t)) \lambda(\omega)
\end{displaymath}
for every  $\omega \in \Omega(\nu)$.
Thus $\lambda\in L^2(\Omega(\nu))$ is an eigenfunction of $T$ to the eigenvalue $\lambda \circ \jmath$. By the Peter--Weyl theorem \cite{S96}, the \textit{continuous} characters provide an orthonormal basis of $L_2(\Omega(\nu))$. In particular, $(\Omega(\nu), \alpha)$ has pure point dynamical spectrum. As every eigenvalue is simple, compare \cite[Ch.~3]{Wal} for the case of $\mathbb Z$-actions, this also shows that any character is continuous.

\smallskip

\noindent (iii) This is an application of the diffraction formula Eqn.~\eqref{eq:FBC1} in Section~\ref{difftheo}.
\qed
\end{proof}

If $G$ is $\sigma$-compact, then both the autocorrelation and the diffraction measure can be computed via the usual limiting procedure.

\begin{corollary}\label{cor:diff}
Assume that $G$ is $\sigma$-compact and fix any van Hove sequence $(B_n)_{n\in \mathbb N}$ in $G$.
Let $\nu\in \MTB$ be strongly almost periodic. Then the following hold.
\begin{itemize}
\item[(i)] The autocorrelation measure
\begin{displaymath}
\gamma = \lim_{n \to \infty} \frac{1}{m_G(B_n)} \, \omega|_{B_n} \ast \widetilde{\omega|_{B_n}}
\end{displaymath}
exists as a vague limit for all $\omega \in \Omega(\nu)$, and the limit $\gamma\in\MTB$ is independent of the choice of $\omega\in\Omega(\nu)$.  The autocorrelation is a positive definite strongly almost periodic measure and hence Fourier transformable as a measure.

\item[(ii)] The diffraction measure  $\widehat{\gamma}\in \mathcal M^\infty(\widehat G)$ is pure point and given by
\begin{displaymath}
\widehat{\gamma} = \sum_{\lambda \in \widehat{\Omega(\nu)}} |c_{\lambda\circ \jmath }|^2\, \delta_{\lambda \circ \jmath}\,.
\end{displaymath}
In the above equation, we have
$c_{\lambda\circ \jmath }(\omega) = \lim_{n \to \infty} c_{\lambda\circ \jmath}^{(n)}(\omega)$, where
\begin{displaymath}
c_{\lambda\circ \jmath}^{(n)} (\omega) = \frac{1}{m_G(B_n)}  \int_{B_n} (\lambda \circ \jmath)(t) \, {\rm d}\omega(t) \,,
\end{displaymath}
and the limit $c_{\lambda\circ \jmath }(\omega)$ has constant modulus, which we abbreviate by $|c_{\lambda\circ \jmath }|$.
\end{itemize}
\end{corollary}

\begin{proof}
Existence of $\gamma$ and independence of the choice of $\omega\in \Omega(\nu)$ is discussed in Section~\ref{difftheo}. Pure point dynamical spectrum of $(\Omega(\nu),\alpha)$, which is shown in Theorem~\ref{stronglyAlmostPeriodic-implies-purePointSpectrum}, implies that $\widehat\gamma$ is a pure point measure, see e.g.~\cite{Martin2}. The explicit expression for the amplitudes follows from Eqn.~\eqref{eq:FBC2} in Section~\ref{difftheo}.
\qed
\end{proof}

\section{Cut-and-project schemes and model sets}\label{cpmodel}

As a preparation for the following sections, we explain the cut-and-project construction. In the physics community, it has been invented to obtain point sets with unusual symmetries \cite{KN} and diffraction  \cite{DK} such as models of quasicrystals, see e.g.~\cite{RS17b} for a recent review.  In fact that construction already emerged in Meyer's work \cite{Meyer} within a harmonic analysis context. Meyer's approach was later popularised by Moody. We will adopt his notation  \cite{Moody-old,Moody-model sets}.

\smallskip

A triple $\cp$ is called a cut-and-project scheme (or simply CPS) if
$G$ and $H$  are LCAG and
$\cL$ is a {\em lattice\/} in $G\times H$ (i.e., a co-compact
discrete subgroup) such that
\begin{itemize}
\item the canonical projection $\pi^G : G\times H \longrightarrow G$ is
one-to-one between $\cL$ and $L:=\pi^G (\cL)$ (in other words,
$\cL \cap (\{0\}\times H) =\{(0,0)\}$), and
\item the image $L^\star = \pi^{H}(\cL)$ of the canonical
  projection  $\pi^{H} : G\times H
\longrightarrow H$ is dense in $H$.
\end{itemize}
The group $H$ is called the {\em internal} space. Given
these properties of the projections $\pi^G$ and $\pi^H$, one
can define the so-called star map $(\cdot)^\star\!: L \longrightarrow H$ as follows.
If $x\in L$, then there is a unique $y\in H$ such that $(x,y)\in\cL$, and we set $x^\star=y$.
If we denote the inverse map of $(\pi^G|_\cL)$ by $(\pi^G|_\cL)^{-1}:L\to \cL$, we then
have $x^\star = \big( \pi^H \circ (\pi^G|_\cL)^{-1}\big) (x)$.
The situation is summarized in the following diagram.
\begin{center}
\begin{tikzcd}
G & \arrow[swap]{l}{\pi^G} G\times H \arrow{r}{\pi^H}& H \\
L \arrow[hookrightarrow]{u}  \arrow[swap, bend right=10]{rr}{\star} & \arrow[hookrightarrow]{u} \arrow[swap, bend right=5]{l}{1-1} \cL \arrow[bend left=5]{r}{\pi^H} & \arrow[hookrightarrow, swap]{u}{\text{dense}}L^\star
\end{tikzcd}
\end{center}
We fix Haar measures $m_G$ on $G$ and $m_H$ on $H$. We then denote by $\mathrm{dens}(\cL)$ the  inverse Haar measure of a measurable fundamental domain for $\cL$.  If $G\times H$ is $\sigma$-compact, it equals  the density of lattice points in $G\times H$. Given a CPS,
we can associate to any $W \subset H$, called the {\em window}, the set
\begin{displaymath}
\oplam(W):= \{x \in L : x^{\star} \in W\}\,.
\end{displaymath}
If $W$ is relatively compact, then $\oplam(W)$ is called a weak model set. Any weak model set is uniformly discrete. If in addition $\mathring{W}\ne\varnothing$, then $\oplam(W)$ is called a model set. Any model set is a Delone set, i.e., it is uniformly discrete and relatively dense.

\smallskip

A CPS gives rise to a dynamical system as follows. Define $\TT :=(G \times H) / \cL$.  Then $\TT$ is a compact abelian group by assumption on $\cL$. Let
\[ G \times H \longrightarrow \TT, \;\:(t,k)\mapsto [t,k],\]
be the canonical quotient map.  There is a canonical continuous group
homomorphism
\[ \iota : G\longrightarrow \TT, \;\: t \mapsto [t,0].\]
The homomorphism $\iota$ has dense range as $L^\star$ is dense in $H$. It
 induces an action $\beta$ of $G$ on $\TT$ via
\[\beta : G\times \TT \longrightarrow\TT,\:\; \beta_t([s,k]):= \iota(-t) +  [s,k]= [s - t,k].\]
The dynamical system $(\TT,\beta)$ will play a crucial role in our
considerations  as it appears in the torus parametrization of the hull associated to (a weighted version of) $\oplam(W)$, see Theorem~\ref{nu-stronglyAlmostPeriodic} and \cite{LR,KR}.
The dynamical system $(\TT,\beta)$  is minimal and uniquely ergodic, as $\iota$ has
dense range.  Moreover, it has pure point spectrum. More precisely,
the dual group $\widehat{\TT}$ gives a set of continuous eigenfunctions, which form
a complete orthonormal basis by the Peter--Weyl theorem  (see \cite{S96} and \cite{Martin2} for
further details). By Pontryagin duality, the dual group $\widehat{\TT}$ can be naturally
identified with the dual lattice
\begin{displaymath}
\cL^0=\{(\chi,\eta)\in \widehat G\times \widehat H: \chi(x)\eta(y)=1 \text{ for all } (x,y)\in \cL\}\,.
\end{displaymath}
Here we identify $\lambda\in\widehat \TT$ and $(\chi,\eta)\in \cL^0$ if and only if
$\lambda([t,k])=\chi(t)\eta(k)$ for all $[t,k]\in\TT$. For a CPS $(G,H,\cL)$,
the situation simplifies as $(\widehat G,\widehat H, \cL^0)$ is also a CPS. As the projection
$\pi^{\widehat G}|_{\cL^0}$ is one-to-one, we can write 
\begin{displaymath}
\cL^0=\{(\chi,\chi^\star)\in \widehat G\times
\widehat H: \chi(x)\chi^\star(x^\star)=1 \text{ for all } x\in L\} \ ,
\end{displaymath}
and we may thus identify $\widehat \TT$ with $L^0:=\pi^{\widehat G}(\cL^0)$. The potential Bragg peak positions in the diffraction of $\oplam(W)$ are elements of $L^0$, compare Theorem~\ref{nu-stronglyAlmostPeriodic}.

\section{Lattice modulations}\label{sec:lmd}

We now generalize the example of Section~\ref{sec:sinemod} to (almost periodic) modulations of lattices in LCAG, see Appendix~\ref{app:gvap} for basics about almost periodic functions taking values in a group.
We will show that any modulated lattice is a deformed weighted lattice model set, and vice versa that any deformed weighted lattice model set is a modulated lattice. We will further show that the class of modulations of a lattice is stable under modulation.
In Section~\ref{Smoothly} we will show that general deformed weighted model sets are strongly almost periodic.
This means, in particular, that  any lattice modulation is strongly almost periodic.

\begin{definition}[modulated lattice]\label{def:modlat} Let $L$ be lattice in a LCAG $G$, and let $w:L\to \mathbb C$ and $g: L \to G$ be almost periodic functions. Then the weighted Dirac comb
\[\delta_L^{w,g}=\sum_{\ell\in L} w(\ell)\, \delta_{\ell+g(\ell)}\]
is called a \emph{modulation} of $L$. The functions $w,g$ are called the \emph{modulation functions}.
\end{definition}

\begin{remark}\label{rem:mult}
\begin{itemize}
\item[(i)] We interpret a lattice modulation as a measure. We thus allow for multiplicities, i.e., the above sum may include terms $\ell\ne  \ell'$ satisfying $\ell+g(\ell)=\ell'+g(\ell')$. The measure setting also allows to incorporate a weighting of the lattice points.
\item[(ii)] If the modulation functions are almost periodic, then the resulting lattice modulations are strongly almost periodic. This can be seen from Proposition~\ref{stronglyAlmostPeriodic-of-deltaLambda} and Theorem~\ref{nu-stronglyAlmostPeriodic}. In order to analyse the converse statement, restrict to $G=\mathbb R$ and $w\equiv1$.  In that situation, the converse can be shown to be true if $g$ has sufficiently small range. However the converse may not be true in general, as can be seen from the artificial example $g:\mathbb Z\to \mathbb R$ given by $g(0)=1$, $g(1)=-1$ and $g(n)=0$ otherwise. 
\item[(iii)] As any modulation function is bounded, any modulation is a measure of locally finite support. Some geometric properties of modulated lattices are discussed in \cite[Ex.~3.10]{FR14}.
\end{itemize}
\end{remark}

\subsection{Lattice modulations and deformations}

We argue that modulated lattices are so-called deformed weighted model sets. Within that framework, we have a very good overview of structural properties of the modulated lattice.  In fact, any deformed weighted lattice model set is a modulated lattice. Our arguments use the Bohr compactification of the lattice, see Appendix~\ref{app:gvap}.

\medskip

There is a natural choice for a CPS, which has compact internal space: The triple $(G,H_L,\mathcal L_L)$, where  $H_L:=L_\mathsf{b}$ is the Bohr compactification of $L$, together with the lattice $\mathcal L_L=\{(\ell, i_\mathsf{b}(\ell)):\ell\in L\}$, where $i_\mathsf{b}: L\to L_\mathsf{b}$ is the canonical dense embedding. The lattice is the model set whose window is the whole internal space.

\begin{lemma}\label{lem:Bohrcps}
The above CPS is well-defined.
\end{lemma}

\begin{proof}
We argue that $\mathcal L_L$ is a lattice. It is clear that $\mathcal L_L$ is a group. To show uniform discreteness, take a zero neighborhood $U\subset G$ such that $L\cap U=\{0\}$. Then $\mathcal L_L\cap (U\times H_L)=\{0\}$. Indeed, if $(\ell, i_\mathsf{b}(\ell))\in U\times H_L$, then $\ell=0$, which implies that $i_\mathsf{b}(0)=0$ is the neutral element in $H_L$. Also, $\mathcal L_L$ is relatively dense in $G\times H_L$ as $L$ is relatively dense in $G$ and as $H_L$ is compact. Thus $\mathcal L_L$ is a lattice in $G\times H_L$. By construction, $\pi^{H_L}(\mathcal L_L)=i_\mathsf{b}(L)$ is dense in $H_L$. To see that $\pi^G$ is one-to-one on $\mathcal L_L$, consider any $(\ell_1, i_\mathsf{b}(\ell_1))$ and $(\ell_2, i_\mathsf{b}(\ell_2))$ such that $\ell_1=\ell_2$. Then also $i_\mathsf{b}(\ell_1)=i_\mathsf{b}(\ell_2)$.
\qed
\end{proof}

\medskip

We next explain the notion of deformed weighted model set. Whereas the general definition will be analyzed from Section~\ref{Smoothly} onward, in this section we will consider the special case that the underlying model set is a lattice.

\begin{definition}[Deformed weighted model set]\label{def:dwms}
Let $(G,H,\cL)$ be any CPS. Let $f\in C_c(H)$ and $p\in C(H,G)$ be weight and deformation functions and let $L=\pi^G(\cL)$.
\begin{itemize}
\item[(i)] The weighted Dirac comb
\begin{displaymath}
\nu=\sum_{\ell\in L}  f(\ell^\star)\, \delta_{\ell+p(\ell^\star)}
\end{displaymath}
is called a \textit{deformed weighted model set}  in the CPS $(G,H,\cL)$.

\item[(ii)] If $H$ is compact, then $L$ is a lattice in $G$, as $L=\pi^G(\cL\cap (G\times H))$ is both a model set and a group. In that case we speak of a \textit{deformed weighted lattice model set}.
\end{itemize}
\end{definition}

The following three results show that the class of deformed weighted lattice model sets coincides with the class of lattice modulations.

\begin{proposition} \label{stronglyAlmostPeriodic-of-deltaLambda}
Any lattice modulation is a deformed weighted lattice model set in the CPS of Lemma~\ref{lem:Bohrcps}.
\end{proposition}

\begin{proof}

Use the Bohr compactification $L_\mathsf{b}$ to write $w=w_\mathsf{b}\circ i_\mathsf{b}$ for $w_\mathsf{b}\in C(L_\mathsf{b})$ and $g=g_\mathsf{b}\circ i_\mathsf{b}$ for $g_\mathsf{b}\in C(L_\mathsf{b}, G)$. Define $f\in C_c(H_L)$ by $f=w_\mathsf{b}$, and define $p\in C(H_L,G)$ by $p=g_\mathsf{b}$. We then have
\begin{displaymath}
\sum_{\ell\in L} f(\ell^\star)\,\delta_{\ell+p(\ell^\star)}=\sum_{\ell\in L} w_\mathsf{b}(i_\mathsf{b}(\ell))\,\delta_{\ell+g_\mathsf{b}(i_\mathsf{b}(\ell))}=
\sum_{\ell\in L} w(\ell)\,\delta_{\ell+g(\ell)}\,.
\end{displaymath}
\qed
\end{proof}

\begin{remark}\label{rem:modlat}
The previous proposition states that any modulated lattice is a  deformed weighted model set. As a consequence, any modulated lattice is strongly almost periodic by Theorem~\ref{nu-stronglyAlmostPeriodic}.
\end{remark}

The above CPS is, in some sense, a universal CPS scheme for the lattice $L$. This idea is further elaborated in Section~\ref{sec:MDMS}. Here we note the following result.

\begin{proposition}
Let $(G,H,\cL)$ be a CPS with compact $H$ and lattice $L=\pi^G(\cL)$. Then any deformed weighted lattice model set in $(G,H,\cL)$  is a deformed weighted lattice model set in $(G, H_L, \cL_L)$.
\end{proposition}

\begin{proof}
Let $f\in C_c(H)$ be the weight function and let $p \in C(H,G)$ be the deformation map.
Since $L= \oplam(H)$, the star map ${}^\star : L \to H$ is a continuous homomorphism from $L$ to the compact group $H$. By the universal property of the Bohr compactification, we may find a continuous group homomorphism $\psi : L_\mathsf{b}=H_L \to H$ such that for all $\ell \in L$ we have $\psi( i_b(\ell))=\ell^\star$. Then $f'=f\circ \psi\in C_c(H_L)$ and $p'=p \circ \psi \in C(H_L, G)$, and in the CPS $(G, H_L, \cL_L)$ we have
\begin{displaymath}
\sum_{\ell \in L} f'(i_\mathsf{b}(\ell)) \, \delta_{\ell+p'(i_\mathsf{b}(\ell))}=\sum_{\ell \in L} f(\psi(i_\mathsf{b}(\ell))) \, \delta_{\ell+p(\psi(i_\mathsf{b}(\ell)))}=\sum_{\ell \in L} f(\ell^\star) \, \delta_{\ell+p(\ell^\star)} \ .
\end{displaymath}
\qed
\end{proof}

\begin{proposition}
Any deformed lattice model set in $(G, H_L, \mathcal L_L)$  is a modulation of $L$.
\end{proposition}

\begin{proof}
Assume that the deformed weighted lattice model set has weight function $f\in C_c(H_L)$ and deformation function $p \in C(H_L, G)$. Then the deformed weighted lattice model set can be written as $\nu=\sum_{\ell\in L} f(\ell^\star)\,\delta_{\ell+p(\ell^\star)}$.  Now $f(\ell^\star)=(f\circ i_\mathsf{b})(\ell)$, and $f\circ i_\mathsf{b}: L\to \mathbb C$ is almost periodic by Proposition~\ref{prop:BochBohr}. Also $p(\ell^\star)=(p\circ i_\mathsf{b})(\ell)$, and $p\circ i_\mathsf{b}: L\to G$ is almost periodic by Proposition~\ref{prop:BochBohr}.
Hence $\nu$ is a modulation of $L$ with modulation functions   $w=f\circ i_\mathsf{b}$ and  $g=p\circ i_\mathsf{b}$.
\qed
\end{proof}

We gather the previous results in the following theorem.

\begin{theorem}
The class of modulated lattices coincides with the class of deformed weighted lattice model sets. \qed
\end{theorem}

\medskip

\begin{remark}The above general results have been obtained using the Bohr compactification of the lattice. However when analyzing particular examples, one might construct a ``minimal'' internal space using the modulation function, see the following example. This is in line with the experimental situation, where the number of Miller indices is determined and a torus of appropriate dimension is added to an underlying crystallographic model. This approach has been phrased in terms of model sets in \cite[Sec.~3]{S08}.
\end{remark}

\begin{example}[Sine modulated integers revisited]
Consider the situation of a Dirac comb $\sum_{\ell\in L} \delta_{\ell+g(\ell)}$, i.e., the modulation function $w$ is constant. We will use the group compactification $L_g=\overline{\{T_\ell g^\dag: \ell\in L\}}$ of $L$ by $g^\dag$, where $g^\dag(x)=g(-x)$. Noting that  $\mathcal L_g=\{(\ell, T_\ell g^\dag): \ell\in L\}$ is a lattice in $G\times L_g$, we can consider the CPS $(G,L_g,\mathcal L_g)$. 
Define $f\in C_c(L_g)$ by $f\equiv 1$ and  $p\in C(L_g,G)$ by $p(h)=h(0)$ for $h\in C(L_g,G)$. We then have
\begin{displaymath}
\sum_{\ell\in L} f(\ell^\star) \, \delta_{\ell+p(\ell^\star)}=\sum_{\ell\in L} \delta_{\ell+g(\ell)}\, ,
\end{displaymath}
i.e., the modulated lattice is a deformed weighted lattice model set in the CPS $(G,L_g, \cL_g)$.
This construction was used for sine modulated integers in Section~\ref{sec:hull}.
\end{example}

\subsection{Modulations of modulated lattices}

We show that the class of modulations of a lattice is stable under modulation. We first extend the definition of modulation beyond the lattice case.

\begin{definition}[modulated measure]\label{def:modps} Given $\nu\in \mathcal{M}(G)$, let $w\in C_\mathsf{u}(G)$ and $g\in C_\mathsf{u}(G, G)$ be almost periodic functions. Then the measure $\nu^{w,g}\in\mathcal{M}(G)$, defined by
\begin{displaymath}
\nu^{w,g}(\varphi)=\int w(x) \varphi(x+g(x))\, {\rm d}\nu(x)
\end{displaymath}
for $\varphi\in C_c(G)$, is called a \emph{modulation} of $\nu$. The functions $w, g$ are called the \emph{modulation functions}.
\end{definition}

\begin{remark}
The above indeed extends Definition~\ref{def:modlat} of a modulated lattice. The domain of the modulation functions is chosen to be the group $G$. Almost periodicity is inherited by passing to subgroups, compare Lemma~\ref{almost-periodic-on-G-is-on-L}.
\end{remark}

\begin{proposition}[lattice modulations are modulation stable]\label{prop:mlms}
Let $\delta_L^{w,g}$ be a modulation of $L$, and let $w'\in C_\mathsf{u}(G)$  and $g'\in C_\mathsf{u}(G,G)$ be almost periodic functions. Then $(\delta_L^{w,g})^{w', g'}$ is a modulation of $L$.
\end{proposition}

\begin{proof}
Note first that by definition we have
\begin{displaymath}
(\delta_L^{w,g})^{w',g'} = \left(\sum_{\ell\in L} w(\ell) \, \delta_{\ell+g(\ell)}\right)^{w',g'}
=\sum_{\ell \in L} w(\ell) w'(\ell+w(\ell)) \, \delta_{\ell+g(\ell)+g'(\ell+g(\ell))} \ .
\end{displaymath}
The statement follows if $g''\in C_\mathsf{u}(L, G)$, given by $g''(\ell)=g(\ell)+g'(\ell+g(\ell))$, and $w''\in C_\mathsf{u}(L)$, given by $w''(\ell)=w(\ell)+w'(\ell+w(\ell))$, are almost periodic on $L$. We show this for $g''$, the argument for $w''$ is analogous. Write $g=g_\mathsf{b}\circ i_L$ for $g_\mathsf{b}\in C(L_\mathsf{b}, G)$ and $g'=g'_\mathsf{b}\circ i_G$ for $g'_\mathsf{b}\in C(G_\mathsf{b}, G)$. Using the notation of Lemma ~\ref{rem:com}, we have
\begin{displaymath}
\begin{split}
g''(\ell)&=g(\ell)+g'(\ell+g(\ell))=g_\mathsf{b}\circ i_L(\ell)+g'_\mathsf{b}\circ i_G(\ell+g_\mathsf{b}\circ i_L(\ell)) \\
&=g_\mathsf{b}\circ i_L (\ell)+g'_\mathsf{b}(i_\mathsf{b}\circ i_L(\ell)+i_G\circ g_\mathsf{b}\circ i_L(\ell))\ .
\end{split}
\end{displaymath}
Together with  Lemma~\ref{rem:com} it follows that $g''=g''_\mathsf{b} \circ i_L$ for some $g''_\mathsf{b}\in C(L_\mathsf{b}, G)$. Hence $g''$ is almost periodic on $L$ by Proposition~\ref{prop:BochBohr}.
\qed
\end{proof}

\section{Deformed weighted model sets }\label{Smoothly}

We give a large class of strongly almost periodic measures arising from model sets.
In general, Dirac combs of model sets are not strongly almost periodic. However weighted model sets \cite{BL} and more general dense Dirac combs \cite{LR} are strongly almost periodic.
Here we show that also weighted versions of deformed model sets \cite{BD,BL2} are strongly almost periodic. This comprises the weighted model sets with continuous weight functions of \cite{BM,LR,Ric,NS11}. Also ideal crystals and their modulations, which are discussed  in Section~\ref{The case of ideal crystals}, fall into this class.

\begin{theorem} \label{nu-stronglyAlmostPeriodic}
Consider the setting of Definition~\ref{def:dwms} (i). Then the following hold.
\begin{itemize}
\item[(i)] The deformed weighted model set $\nu \in \cM^\infty(G)$ is strongly almost periodic.

\item[(ii)] The hull $\Omega(\nu)$ is a compact abelian group, and the canonical  action $\alpha$ is  continuous and onto. In particular,  $(\Omega(\nu), \alpha)$ is minimal, uniquely ergodic with continuous eigenfunctions. The unique ergodic measure is the Haar measure on the group, and the set of continuous eigenfunctions is the group dual to $\Omega(\nu)$.
 \item[(iii)] Consider $\TT=(G\times H)/\cL$ together with canonical $G$-action $\beta$. The topological dynamical system $(\varOmega (\nu),\alpha)$ is a factor of $(\TT,\beta)$ with factor map $ \mu : \TT \to \varOmega(\nu)$ given by
\begin{displaymath}
\mu ([s,k]) (\varphi) =\sum_{\ell\in L} f(\ell^\star+k) \,\varphi (\ell+s + p(\ell^\star+k))
\end{displaymath}
for any $\varphi\in C_c(G)$. In fact $\mu $ is a group homomorphism.

\item[(iv)] The diffraction measure of any $\omega\in \Omega(\nu)$ is given by $\widehat\gamma=\sum_{\chi \in L^0} |a_\chi|^2\delta_\chi$, where
\begin{displaymath}
a_\chi= \mathrm{dens}(\cL) \cdot \int_H \, \overline{\chi^\star(y)} \,\chi(p(y)) f(y) \, {\rm d}m_H(y) \,.
\end{displaymath}
Here $\mathrm{dens}(\cL)$ is the inverse measure of a measurable fundamental domain for the lattice $\cL$. In the $\sigma$-compact case, it coincides with the density of lattice points in $G\times H$.
\end{itemize}
\end{theorem}

\begin{remark}
\begin{itemize}
\item [(i)] The continuous map $\mu:\TT\to \Omega(\nu)$ may be called the torus parametrization of $\Omega(\nu)$. For undeformed weighted model sets, it has previously been used in \cite{LR}.
\item [(ii)] Consider (unweighted and undeformed) model set Dirac combs. In that case,  continuity of $\mu$ fails in general. A torus parametrization is then often defined as a natural continuous map  $\Omega(\nu)\to\TT$, which is a partial inverse of $\mu$.
On the other hand, $\mu$ is measurable in that case, which is sufficient in order to infer diffraction properties \cite{KR,KR17}.

\item[(iii)] The proofs of parts $(i), (ii), (iii)$ adapt arguments from \cite[Sec.~7]{LR}. With more technical effort, the above setting could be extended to comprise deformations of dense Dirac combs \cite{LR}.

\item[(iv)] The formula for the diffraction measure in $(iv)$ shows that, in comparison to the undeformed model set, the potential Bragg peak positions are unchanged, but their amplitudes are altered. In the Euclidean setting, the above result has been obtained in \cite{BD}, by calculating the autocorrelation of the deformed weighted model set and its transform using Eqn.~\eqref{eq:FBC2}. We will give a proof based on almost periodicity and dynamical systems, which uses Eqn.~\eqref{eq:FBC1}. 

\item[(v)]  If there are no extinctions (i.e., zero amplitudes at potential Bragg peak positions), then the result discussed in (iv) says that both hulls are spectrally isomorphic. In that case both hulls are also measure-theoretically isomorphic, by the Halmos--von Neumann theorem (see e.g.~\cite{Wal} for the case of $\mathbb Z$-actions). Hence dynamical concepts  stronger than measure-theoretic isomorphism are needed in order to distinguish the different types of diffraction.

\item[(vi)] The statements in $(iii)$ and $(iv)$ are also proved in \cite[Thm.~6]{BL2} and in \cite[Sec.~7.1]{Len} by a different method. Diffraction calculations for deformed (unweighted) model sets also appear in \cite{BD,BL2}. In particular, there are explicit calculations for the Fibonacci chain \cite[Sec.~5]{BD} and the silver mean chain \cite[Sec.~7]{BL2}.

\end{itemize}

\end{remark}

\begin{proof}

\noindent (i)
We show that $\nu$ is strongly almost periodic. This implies translation boundedness, compare Section~\ref{Strongly}.
Consider any $\varphi\in C_c(G)$. Note first that for any $y \in G$ and for any $t\in L$ we have by a standard estimate
\begin{displaymath} \label{strongly-almost-periodic-inequality}
|(\nu * \varphi)(y) - (\nu * \varphi)(y-t)|  \le
\sum_{\ell \in L} |f(\ell^{\star})\varphi(y-\ell-p(\ell^{\star})) - f(\ell^{\star} - t^{\star})\varphi(y-\ell-p(\ell^{\star}-t^{\star}))| \,.
\end{displaymath}
Now take any compact zero neighborhood $U$ in $H$ and restrict to $t\in L$ such that $t^\star\in U$, i.e., restrict to $t\in\oplam(U)$. Then we may have non-vanishing summands only if $\ell^\star\in \supp(f) \cup (\supp(f)+U)=:W$, which is a compact set. In addition to the previous condition, we may have non-vanishing summands only if $\ell\in y-p(W)-\supp(\varphi)=:y+K$, which is a compact set. Hence the above sum may be restricted without loss of generality to $\ell\in\oplam(W)\cap (y+K)$, which is a finite set. Using the triangle inequality and standard estimates, we thus get  for any $t \in \oplam(U)$ the estimate
\begin{displaymath}
\begin{split}
|(\nu * \varphi)(y) &- ( \nu * \varphi)(y-t)|
\le  \sum_{\ell \in \oplam(W)\cap (y+K)}  |f(\ell^{\star}) - f(\ell^{\star} - t^{\star})| \,\|\varphi\|_\infty \, +\\
&+ \sum_{\ell \in \oplam(W)\cap (y+K)}  \|f\|_\infty\, |\varphi(y-\ell-p(\ell^{\star})) - \varphi(y-\ell-p(\ell^{\star}-t^{\star}))| \,.
\end{split}
\end{displaymath}
Now consider the first sum in the above estimate. As $\oplam(W)$ is uniformly discrete, the number of terms in the sum is bounded uniformly in $y\in G$. As $f$ is uniformly continuous, this implies that the first sum gets arbitrarily small if $t^\star$ approaches zero, uniformly in $y\in G$. We can argue similarly for the second sum, using that $\varphi$ is uniformly continuous and that $p$ is uniformly continuous when restricted to the compact set $W$. Hence if arbitrary $\varepsilon>0$ is given, we find a zero neighborhood $V\subset U$ such that
\begin{displaymath}
\begin{split}
|(\nu * \varphi)(y) &- ( \nu * \varphi)(y-t)| \le \varepsilon
\end{split}
\end{displaymath}
for all $y\in G$ and for every $t\in \oplam(V)$. As $\oplam(V)$ is relatively dense and as $\varphi$ was arbitrary, this shows that $\nu$ is strongly almost periodic.

\smallskip

\noindent (ii) This follows from (i) and Theorem~\ref{stronglyAlmostPeriodic-implies-purePointSpectrum}.

\smallskip

\noindent (iii) We first show that $\mu:\TT\to \Omega(\nu)$ is well-defined and continuous.  Consider the map $\mu'(s,k):C_c(G)\to \mathbb C$, given by
\begin{displaymath}
\varphi \mapsto \mu'(s,k)=\sum_{\ell\in L} f(\ell^\star+k) \,\varphi (\ell+s + p(\ell^\star+k)) \,.
\end{displaymath}
As $\mu'$ is $\cL$-invariant, the map $\mu$ is well-defined.
The map $\mu'$ is continuous: Let $(\xi_\iota)_{\iota\in I}$ be any net in $G\times H$ such that $\xi_\iota\to (s,k)$. Writing $\xi_\iota=(s_\iota,k_\iota)$, we have $s_\iota\to s$ and $k_\iota\to k$ as the canonical projections are continuous.   We show $\mu'(s_\iota,k_\iota)(\varphi)\to \mu'(s,k)$. As the factor map is continuous, this implies continuity of $\mu$. Choose a compact neighborhood $W_k\subset H$ of $\supp(f)-k$ and choose a compact neighborhood $K_s\subset G$ of $\supp(\varphi)-p(\supp(f))-s$. Then there is $\iota_0 \in I$ such that $\supp(f)-k_\iota\subset W_k$ and $\supp(\varphi)-p(\supp(f))-s_\iota\subset K_s$ for all $\iota > \iota_0$. Using the triangle inequality and standard estimates, we obtain for $\iota > \iota_0$ the estimate
\begin{displaymath}
\begin{split}
|\mu'(s,k)(\varphi)&-\mu'(s_\iota,k_\iota)(\varphi)|\le
 \sum_{\ell \in \oplam(W_k)\cap K_s}  |f(\ell^{\star}+k) - f(\ell^{\star} +k_\iota)| \,\|\varphi\|_\infty \, +\\
&+ \sum_{\ell \in \oplam(W_k)\cap K_s}  \|f\|_\infty\, |\varphi(\ell+s+p(\ell^{\star}+k)) - \varphi(\ell+s_\iota+p(\ell^{\star}+k_\iota))| \,.
\end{split}
\end{displaymath}
As the above sums are finite, the rhs tends to zero with increasing $\iota$ as $f$, $\varphi$ and $p$ are continuous.  This shows continuity of $\mu'$ and of $\mu$, as the projection map is continuous.
By definition, we have $\alpha_t(\mu([s,k]))=\mu(\beta_t([s,k]))$
for every $t\in G$ and every $(s,k)\in G\times H$.
In particular we have
\begin{equation}\label{eq:mu0}
\mu(\beta_t([0,0]))=\alpha_t(\mu([0,0]))=\alpha_t(\mu'(0,0))=\alpha_t\nu\,.
\end{equation}
By minimality of  $\TT$, continuity of $\mu$ and compactness of  $\TT$ we thus have
\begin{displaymath}
\mu(\TT)=\mu(\overline{\{\beta_t([0,0]):t\in G\}})=\overline{\{\alpha_t\nu:t\in G\}}=\Omega(\nu)\,.
\end{displaymath}
We have shown that $\mu:\TT\to\Omega(\nu)$ is continuous, onto and commutes with the translation action. Hence $\mu$ is a factor map.
Let the group homomorphism  $\jmath : G\longrightarrow \Omega (\nu), t\mapsto \alpha_t \nu$ be given by Lemma \ref{lem:sap}. Then Eqn.~\eqref{eq:mu0} gives that $\jmath(t) = \alpha_t \nu = \mu(\iota(t))$. Using Lemma~\ref{lem:sap}, we thus have
\begin{displaymath}
\mu (\iota (t)\dot{+} x) = \jmath(t)\dot{+}  \mu (x) = \mu(\iota(t))\dot{+} \mu (x)
\end{displaymath}
for all $t\in G$ and $x\in \TT$. Thus $\mu$ is a group homomorphism by minimality.

\smallskip

\noindent (iv) By (iii), the diffraction measure $\widehat \gamma$ of $\nu$ is pure point. As the torus para\-metrization map $\mu:\TT\to\Omega(\nu)$ is a factor map, it can be used to embed $L^2(\Omega(\nu))$ isometrically into $L^2(\TT)$, compare \cite[Thm.~1]{BL2}. Hence we can compute the diffraction using Eqn.~\eqref{eq:FBC1}.  The normalized $L^2(\TT)$-eigenfunctions of the translation $\beta$ are given by
 $e_\chi=(\chi, \chi^\star)\in \widehat\TT \sim \cL^0$ and can be parametrized by their eigenvalue $\chi\in L^0$.
Now by an application of Weil's formula \cite[Eqn.~(3.3.10)]{Rei2}, also called the quotient integral formula \cite[Thm.~1.5.2]{DE}, we get
\begin{displaymath}
\begin{split}
a_\chi&=\int_\TT \overline{e_\chi([t,k])} \mu([t,k])(\chi\cdot\varphi) {\rm d}([t,k]) \\
&= \mathrm{dens}(\cL) \cdot\int_G\int_H \overline{e_\chi(x,y)}  f(y) \chi(x+p(y))\varphi(x+p(y)) {\rm d}(x,y)\\
&= \mathrm{dens}(\cL) \cdot\int_H\chi(p(y)) \overline{  \chi^\star(y)}  f(y){\rm d}y \ .
\end{split}
\end{displaymath}
The factor $ \mathrm{dens}(\cL)$ reflects the non-canonical choice of Haar measure on $\TT$.
\qed
\end{proof}

\section{Modulations of deformed weighted model sets}\label{sec:MDMS}

In this section, we will prove that any modulation of a deformed weighted model set is a deformed weighted model set. We will also prove that, in Euclidean space, the class of deformed  weighted model sets coincides with the class of modulated weighted model sets.

\subsection{An extended CPS for modulations} Let $(G, H_0,\mathcal L_0)$ be any CPS. The Bohr compactification of $G$ can be used to expand the internal space $H_0$. In the resulting CPS $(G,H,\mathcal L)$, the lattice projects injectively to $H$. We will use that scheme in order to describe modulations of a deformed weighted model set.

\begin{lemma}[Extended CPS]\label{lem:excp} Let $(G,H_0,\cL_0)$ be any CPS. Then there exists an extended CPS $(G, H, \mathcal L)$ where $\pi^H|_{\mathcal L}$ is one-to-one with the following property:  Every \{weak, deformed, weighted\} model set in $(G,H_0,\mathcal L_0)$ is a \{weak, deformed, weighted\} model set in $(G, H, \mathcal L)$.
\end{lemma}

\begin{proof}
Consider the Bohr compactification $G_\mathsf{b}$ of $G$ together with the completion homomorphism $i_\mathsf{b}: G \to G_\mathsf{b}$. Note that $i_\mathsf{b}$ is one-to-one, since $G$ is abelian. Setting $H'=H_0\times G_\mathsf{b}$, we extend the star map ${}^{\star|_{H_0}}$ from $H_0$ to $H'$ by defining $\ell^\star=(\ell^{\star|_{H_0}},i_\mathsf{b}(\ell))$ for $\ell\in \pi^G(\mathcal L_0)=L_0$. We claim that the subgroup $\mathcal L=\{(\ell, \ell^\star): \ell\in L_0\}$ of $G\times H'$ is a lattice. To show discreteness of $\mathcal L$ in $G\times H'$, choose open $U_0\subset G\times H_0$ such that $\mathcal L_0\cap U_0=\{0\}$, define $U=U_0\times G_\mathsf{b}$ and consider $(\ell,\ell^\star)\in \mathcal L \cap U$. Then $\ell=0$ due to discreteness of $\mathcal L_0$, which implies $\ell^\star=0$. Since $U$ is non-empty open, it follows that $\mathcal L$ is discrete.
But $\mathcal L$ is also relatively dense in $G \times H'$. To see this, take compact $K_0 \subset G\times H_0$ such that $\mathcal L+K_0=G\times H_0$. Then $K_0\times G_\mathsf{b}$ is compact, and we have
\begin{displaymath}
\mathcal L+(K_0\times G_\mathsf{b}) = (\mathcal L_0+K_0)\times G_\mathsf{b}=G\times H'.
\end{displaymath}
Injectivity of $\pi^G|_{\mathcal L}$ is inherited from injectivity of $\pi^G|_{\mathcal L_0}$.
If we now define $H=\overline{\pi^{H'}(\mathcal L)}\subset H'$, the lattice $\mathcal L$ projects densely to $H$ by definition.  Hence $(G, H,\mathcal L)$ is a CPS. The above formulae can also be used to infer injectivity of $\pi^H$ on $\mathcal L$ from injectivity of $i_\mathsf{b}$.

Now let $\oplam(W_0)$ be a weak model set in $(G,H_0,\cL_0)$. Note that $W=W_0\times G_\mathsf{b}$ is relatively compact and that
\begin{displaymath}
\begin{split}
\oplam(W_0)&=\{ \pi^G(z_0) \,|\, z_0\in \mathcal L_0, \pi^{H_0}(z_0)\in W_0\}\\
&=\{ \pi^G(z) \,|\, z\in \mathcal L, (\pi^{H_0}(z),i_\mathsf{b}(z))\in W_0\times G_\mathsf{b}\}=\oplam(W) \,.
\end{split}
\end{displaymath}
Assume that $\nu=\sum_{\ell\in L_0}f_0(\ell^{\star|_{H_0}}) \delta_{\ell+p_0(\ell^{\star|_{H_0}})}$ is a deformed weighted model set in $(G, H_0,\mathcal L_0)$, where $f_0\in C_c(H_0)$ and $p_0\in C(H_0, G)$. Define $f\in C_c(H)$ by $f(y)=f_0(\pi^{H_0}(y))$ and $p\in C(H,G)$ by $p(y)=p_0(\pi^{H_0}(y))$. As we have $L_0=L=\pi^G(\mathcal L)$, we conclude $\nu=\sum_{\ell\in L}f(\ell^{\star}) \delta_{\ell+p(\ell^\star)}$ is a deformed weighted model set in $(G,H,\cL)$.
\qed
\end{proof}

\subsection{Modulations of deformed weighted model sets}

The class of deformed weighted model sets is stable under modulations.
The argument is analogous to that of Proposition~\ref{prop:mlms}.

\

\begin{theorem}[Deformed weighted model sets are modulation stable]\label{thm:ms}
Let $(G,H,\mathcal L)$ be an extended CPS as in the previous subsection. Then any modulation of a deformed weighted model set in $(G,H,\mathcal L)$ is a deformed weighted model set in $(G,H,\mathcal L)$.
\end{theorem}

\begin{proof}
Consider $H\subset H'=H_0\times G_\mathsf{b}$, where we use the notation of the previous subsection.
Write $L=\pi^G(\cL)$. Let $\sum_{\ell\in L}f(\ell^\star)\delta_{\ell+p(\ell^\star)}$ be any deformed weighted model set in $(G,H,\mathcal L)$, where $f\in C_c(H)$ and where $p\in C(H,G)$.   Let $w\in C_\mathsf{u}(G)$ and $g \in C_\mathsf{u}(G,G)$ be almost periodic functions for the modulation. Choose $g_\mathsf{b}\in C(G_\mathsf{b},G)$ such that $g=g_\mathsf{b}\circ i_\mathsf{b}$. We then have for $\ell\in L$ the equality
\begin{displaymath}
p(\ell^\star)+g(\ell+p(\ell^\star))=p(\ell^\star)+(g_\mathsf{b}\circ i_\mathsf{b})(\ell+p(\ell^\star)))= p(\ell^\star)+g_\mathsf{b}(\pi^{G_\mathsf{b}}(\ell^\star)+i_\mathsf{b}(p(\ell^\star)))\,.
\end{displaymath}
We can now use the rhs of the above equation to define $p'\in C(H,G)$ by $p'(y)=p(y)+g_\mathsf{b}(\pi^{G_\mathsf{b}}(y)+i_\mathsf{b}(p(y)))$.
Similarly, we choose $w_\mathsf{b}\in C(G_\mathsf{b})$ such that $w=w_\mathsf{b}\circ i_\mathsf{b}$, and we get
\begin{displaymath}
w(\ell+p(\ell^\star))=(w_\mathsf{b}\circ i_\mathsf{b})(\ell+p(\ell^\star)))= w_\mathsf{b}(\pi^{G_\mathsf{b}}(\ell^\star)+i_\mathsf{b}(p(\ell^\star)))\,.
\end{displaymath}
Let us define $f'\in C_c(H)$ by $f'(y)=f(y)\cdot w_\mathsf{b}(\pi^{G_\mathsf{b}}(y)+i_\mathsf{b}(p(y)))$.
We then have
\begin{displaymath}
\sum_{\ell\in L}w(\ell+p(\ell^\star)) f(\ell^\star) \, \delta_{\ell+p(\ell^\star)+g(\ell+p(\ell^\star))}=\sum_{\ell\in L}f'(\ell^\star)\delta_{\ell+p'(\ell^\star)}\,.
\end{displaymath}
\qed
\end{proof}

\subsection{Modulations and deformations}

Specializing Theorem~\ref{thm:ms} to the case $p\equiv0$, we get that modulations of  weighted model sets are deformed weighted model sets.

\begin{corollary}
Let $(G, H_0, \cL_0)$ be any CPS, and let $(G, H, \cL)$ be the extended CPS of Lemma~\ref{lem:excp}, where $H\subset H_0 \times G_\mathsf{b}$. Then for any two almost periodic functions $w\in C_\mathsf{u}(G)$ and $g\in C_\mathsf{u}(G, G)$ there exist $f\in C_c(H)$ and $p \in C(H, G)$ such that
\begin{displaymath}
w(\ell)=f(\ell^\star) \ , \qquad  g(\ell)=p(\ell^\star)
\end{displaymath}
for all $\ell \in \pi^G(\cL)$.
In particular, any modulation of a weighted model set in $(G,H_0, \cL_0)$ is a deformed weighted model set in $(G,H, \cL)$. \qed
\end{corollary}

The following result states that for $G=\mathbb R^d$ the converse also holds. Hence in Euclidean space the class of modulated weighted model sets coincides with the class of deformed weighted model sets.

\begin{proposition} Let $(G, H, \cL)$ be any CPS for $G=\RR^d$. Conisder $f\in C_c(H)$ and  $p\in C(H,G)$. Then there exist almost periodic functions $w\in C_\mathsf{u}(G)$ and $g\in C_\mathsf{u}(G, G)$ such that
\begin{displaymath}
w(\ell)=f(\ell^\star) \ , \qquad  g(\ell)=p(\ell^\star)
\end{displaymath}
for all $\ell \in \oplam(W)$ where $W=\mathrm{supp}(f)$. In particular, any deformed weighted model set in $(G,H, \cL)$ is a modulation of a weighted model set in $(G,H, \cL)$.
\end{proposition}

\begin{proof}
We construct $g$ by exploiting the $\RR$-vector space structure of $\RR^d$. We may assume without loss of generality that $p: H\to G$ is constant $\vec{0}$ outside a precompact neighborhood $W_0$ of $W$. Indeed, pick some $h \in C_c(H)$ such that $h \equiv 1$ on $W$, and define $p': H \to G$ via $p'(y)=h(y)\cdot p(y)$. Then $p$ vanishes outside the precompact set $W_0= \{ y \in H : h(y) \neq 0 \}$, and we can replace $p$ by $p'$.  In particular, $p$ can be assumed to be uniformly continuous without loss of generality.

Next, take any zero precompact neighborhood $V_0\subset H$ and take any zero neighborhood $U_0$ in $G$ such that the regular model set $\oplam(W_0+V_0)$ is $U_0$-uniformly discrete. This is possible as $W_0+V_0$ is precompact, and hence $\oplam(W_0+V_0)$ is uniformly discrete.

Finally, choose $\varphi\in C_c(G)$ such that $\varphi(\vec{0})=1$, $0\le \varphi \le 1$ and $\mathrm{supp}(\varphi)\subset U_0$ and define $g:G\to G$ by
\begin{displaymath}
g(x)=\sum_{\ell \in \oplam(W_0+V_0)} \varphi(x-\ell) \cdot p(\ell^\star) =\sum_{\ell \in \oplam(W_0)} \varphi(x-\ell) \cdot p(\ell^\star)\ .
\end{displaymath}
For given $x\in G$, the above sum consists of at most one non-zero term. By definition $g$ is uniformly continuous, and we have $g(\ell)=p(\ell^\star)$ for all $\ell\in\oplam(W)$. Also note that $g(x)=\vec{0}$ if $x\notin \oplam(W_0)+U_0$ as $p(y)=\vec{0}$ for all $y \notin W_0$.  We show that $g$ is almost periodic by showing that its almost periods are relatively dense. Fix arbitrary $\varepsilon>0$ and choose any symmetric zero neighborhood $V\subset H$ such that $V\subset V_0$ and $\| p(y)-p(y')\|<\varepsilon$ whenever $y-y'\in V$, which is possible due to uniform continuity of $p$. Here $\|\cdot\|$ is any norm on $\RR^d$.  We show that $\oplam(V)$ is a set of $\varepsilon$-almost periods for $g$, which proves the claim.

Fix any $t\in \oplam(V)$ and assume first $x\in \oplam(W_0)+U_0$. Then $x=\ell+u$ for some $\ell\in \oplam(W_0)$ and some $u\in U_0$. Moreover we have $x-t=\ell-t+u$ and $\ell-t\in \oplam(W_0+V_0)$. Hence we can estimate
\begin{displaymath}
\|g(x-t)-g(x)\|=\|\varphi(u)\cdot p(\ell^\star-t^\star)-\varphi(u)\cdot p(\ell^\star)\| \le \|p(\ell^\star-t^\star)-p(\ell^\star)\|< \varepsilon \ .
\end{displaymath}
Similarly, assume $x-t\in \oplam(W_0)+U_0$. Then $x-t=\ell+u$ for some $\ell\in \oplam(W_0)$ and some $u\in U_0$. Moreover we have $x=\ell+t+u$ and $\ell+t\in \oplam(W_0+V_0)$. Hence we can estimate
\begin{displaymath}
\|g(x-t)-g(x)\|\le \|p(\ell^\star)-p(\ell^\star+t^\star)\|< \varepsilon \ .
\end{displaymath}
We are left with the case $x\notin \oplam(W_0)+U_0$ and $x-t\notin \oplam(W_0)+U_0$. But then $\|g(x-t)-g(x)\|=\|\vec{0}-\vec{0}\|=0$.

The function $w$ can be constructed analogously. In fact, for $w$ the above construction works on general LCAG.
\qed
\end{proof}

Let us summarize the above two results in the following theorem.

\begin{theorem}
Every modulation of a weighted model set is a deformed weighted model set in the extended CPS from Lemma~\ref{lem:excp}. In Euclidean space, the class of deformed weighted model sets coincides with the class of modulated weighted model sets. \qed
\end{theorem}

\section{Ideal crystals and their modulations}\label{The case of ideal crystals}

Ideal crystals are mathematical models of physical crystals. Given a physical crystal,  the fundamental problem of crystallography is reconstructing its period group and a decoration of the fundamental domain from diffraction data.

\medskip

In this section we will show that ideal crystals and their modulations have a strongly almost periodic Dirac comb. Ideal crystals are model sets having a discrete internal space. We also show that the class of ideal crystals is stable under so-called commensurate modulation.

\begin{definition}\label{def:idealcrystal}
A locally finite set is an \textit{ideal crystal} if its period group is a lattice.
\end{definition}

The \textit{period group} $P(\Lambda)$ of $\Lambda\subset G$ is given by $P(\Lambda)=\{t\in G: t+\Lambda=\Lambda\}$. The above definition is in line with \cite[Def.~5.1]{KL13} and \cite[Def.~1.3]{Lag}, see also the following lemma. An ideal crystal is called a \textit{crystallographic} point set in \cite[Def.~3.1]{BG2}.

\subsection{Parametrization}\label{sub:para}

The following characterization of ideal crystal is certainly well-known, see e.g.~\cite[Prop.~3.1]{BG2}. We provide the short proof for the reader's convenience, as the argument shows how to construct a torus parametrization for the hull of an ideal crystal.

\begin{lemma}
Let $\Lambda\subset G$ be locally finite. Then $\Lambda$ is an ideal crystal if and only if there exist a lattice $\Gamma$ and a finite set $F$ such that $\Lambda=\Gamma+F$.
\end{lemma}

\begin{proof}
``$\Rightarrow$'':
Assume that the period group $P$ of $\Lambda$ is a lattice. Let $D$ be any relatively compact fundamental domain of $P$. Choosing representatives within $D$, we get
\begin{displaymath}
\Lambda=\bigcup_{\lambda\in\Lambda}(P+\lambda)=
\bigcup_{\lambda\in\Lambda\cap D}(P+\lambda)=P+(\Lambda\cap D)
\end{displaymath}
As $\Lambda\cap D$ is a finite set due to local finiteness of $\Lambda$, this shows $\Lambda=\Gamma+F$ with $\Gamma=P$ and $F=\Lambda\cap D$.

\noindent ``$\Leftarrow$'':  Note that the period group $P$ of $\Lambda=\Gamma+F$ satisfies $\Gamma \subset P\subset \Gamma+F-F$. Hence the group $P$ is both discrete and relatively dense. This means that $P$ is a lattice. In particular $\Lambda$ is an ideal crystal.
\qed
\end{proof}

\begin{remark}\label{rem:dsd}
The decomposition $\Lambda =\Gamma+F$ may not be unique. A simple example is $\Lambda=\ZZ+\{0,1/2\}=1/2\cdot \ZZ$, which satisfies $\Lambda=P(\Lambda)$.
In general, we  have $\Gamma\subset P(\Lambda)$, and the above proof shows that we may choose $\Gamma=P(\Lambda)$ and $F=\Lambda \cap D$, where $D$ is a relatively compact fundamental domain of $P(\Lambda)$ in $G$. In that case, $\Lambda$ is uniquely parametrized by $\Gamma$ and $F$, i.e., we have the direct sum decomposition $\Lambda=\Gamma\oplus F$, compare \cite[Prop.~3.1]{BG2}. That decomposition is canonical from an experimental viewpoint.
\end{remark}

It follows from the previous remark that we can parametrize the hull of $\Lambda$ by the compact group $G/\Lambda(G)$. Indeed, $x+P(\Lambda)\mapsto x+\Lambda$ provides a homeomorphism which commutes with the natural shift action. Thus $\delta_\Lambda$ is strongly almost periodic by Lemma~\ref{lem:sap}, and the Bragg peaks in the diffraction spectrum are indexed by the dual of $P(\Lambda)$.

\smallskip

Strong almost periodicity of ideal crystals was already proved in \cite[Cor.~5.6]{KL13} by a different method. That reference also gives a partial converse. Recall that $\Lambda \subset G$ has finite local complexity if $\Lambda-\Lambda$ is uniformly discrete.

\begin{proposition}\cite[Cor.~5.6]{KL13}\label{prop:charidcryst}
Let $\Lambda$ be a Delone set in a compactly generated LCAG $G$. If $\Lambda$ has finite local complexity and $\delta_\Lambda$ is strongly almost periodic, then $\Lambda$ is an ideal crystal.
\qed
\end{proposition}

\subsection{Ideal crystals and model sets}

Ideal crystals can be characterized as model sets having discrete internal space. The parametrization as a model set can be used to compute the diffraction of an ideal crystal by the diffraction formula for model sets.

\begin{proposition}[Ideal crystals and model sets]\label{lem:icasms}

\begin{itemize}
\item[(i)] Any model set from a CPS with discrete internal space is an ideal crystal.
\item[(ii)] Any ideal crystal is a model set in some CPS with discrete internal space.
\end{itemize}
\end{proposition}

\begin{proof}

``(i)'' Consider any CPS $(G,H,\cL)$ with discrete $H$ and relatively compact non-empty window $W\subset H$.
The group $\Gamma= \oplam(\{ 0\})$ is a model set as $H$ is discrete. Thus $\Gamma$ is a Delone set and hence a  lattice in $G$. Moreover, for all $t \in \Gamma$ we have
\[
t+\oplam(W)=\oplam(t^\star+W)=\oplam(0+W)=\oplam(W) \,.
\]
Therefore $\Gamma \subset P(\oplam(W)) \subset \oplam(W)-\oplam(W)$. This shows that $ P(\oplam(W))$ is relatively dense and uniformly discrete, thus a lattice. As $\oplam(W)$ is locally finite, this shows that $\oplam(W)$ is an ideal crystal.

``(ii)''
Assume that $\Lambda=\Gamma+F$ for some lattice $\Gamma$ and some non-empty finite set $F$.
Consider the group $H = G/\Gamma$ equipped with the discrete topology and let
$\cL= \{ (x, \overline{x}) : x \in G \}$, where $\overline{x}= x+\Gamma$.  We claim that $(G, H, \cL)$ is a CPS.

Let us first prove that the group $\cL$ is a lattice.
The group $\cL$ is discrete: Since $\Gamma$ is a lattice in $G$, we can find an open zero neighborhood $U \subset G$ such that $\Gamma \cap U =\{ 0\}$. Since $\{ \overline{0} \}$ is open in $H$ by discreteness of $H$, the set $U \times \{ \overline{0} \}$ is an open zero neighborhood in $G \times H$. Now consider any point  $(y, \overline{y}) \in \cL \cap (U \times \{ \overline{0} \})$. Then $\overline{y}=\overline{0}$, which implies $y \in \Gamma$. Thus $y \in \Gamma \cap U$, which implies $y=0$. Thus $\cL$ is discrete.
The group $\cL$ is relatively dense: Let $K \subset G$ be compact such that $\Gamma +K =G$.
Let $(x, \overline{y}) \in G \times H$ be arbitrary. We can write $x = z_1+ k_1, y=z_2+k_2$ for some $z_i \in \Gamma$ and $ k_i\in K$. Note that $\overline{x}=\overline{k_1}$ and $\overline{y}=\overline{k_2}$. Then
\begin{displaymath}
(x, \overline{y})= (z_1+k_1, \overline{k_2})= (z_1+k_2, \overline{z_1+k_2}) +(k_1-k_2, \overline{0}) \in \cL + \left((K-K) \times \{\overline 0 \}\right) \ .
\end{displaymath}
This shows that $\cL$ is relatively dense. We have shown that $\cL$ is a lattice.

It is obvious from the definition of $\cL$ that $\pi^G|_\cL$ is one-to-one, and that $\pi^H(\cL)=H$ is dense in $H$. Hence $(G,H,\cL)$ is a CPS.

Now let $W=\overline{F}$. Then $W$ is compact as $F$ is finite. Furthermore we have $\oplam(W)=\{x\in G: \overline{x}\in\overline{F}\}=F+\Gamma$.
\qed
\end{proof}

Consider an ideal crystal $\Lambda=\Gamma+F=\oplam(W)$ as in the proof of (ii). Then $P(\Lambda)^\star=P_W$, with $P_W=\{h\in H: h+W=W\}$ the period group of the window. As a simple consequence $P(\Lambda)=\Gamma$ if and only if the window is aperiodic, i.e., $P_W=\{\overline{0}\}$. Hence any ideal crystal can be realized as a model set with aperiodic window.
In that case, the torus parametrization of Theorem~\ref{nu-stronglyAlmostPeriodic} is an isomorphism by \cite[Prop.~13.2]{LR}, which is in line with the results from Section~\ref{sub:para}. Injectivity properties of the torus parametrization in the case of nonzero window periods are discussed in \cite{KR17}.

\subsection{Modulations and deformations of ideal crystals}

As any ideal crystal is a model set, we can consider weighted and deformed variants of it. As the internal space can be chosen to be discrete, the following result is readily inferred from Section~\ref{sec:MDMS}.

\begin{theorem}
Let $\Lambda$ be an ideal crystal, described as a model set in the extended CPS $(G,H, \cL)$ from Lemma~\ref{lem:excp}. Then any modulation of $\Lambda$ is a deformed weighted variant of  the model set $\Lambda$ in $(G,H,\cL)$.  In particular,  any modulation of $\Lambda$ is strongly almost periodic.    In Euclidean space, the class of modulations of $\Lambda$ coincides with the class of deformed weighted variants of the model set $\Lambda$ in $(G,H, \cL)$. The classes of deformed weighted model sets in $(G,H,\cL)$ and of modulations of $\Lambda$ are both modulation stable. \qed
\end{theorem}

\smallskip

The class of ideal crystals is stable under a restricted class of modulations, which is sometimes called \textit{commensurate modulations}. A function on $G$ is called \textit{fully periodic} or \textit{crystallographic}, if its period group is a lattice in $G$.

\begin{proposition}\label{prop:icms}
Let $\Lambda=\Gamma+F$ be an ideal crystal, and let $g:G\to G$ be any function such that $g|_\Gamma$ is fully periodic. Then $\Lambda^g=\{\lambda+g(\lambda): \lambda\in \Lambda\}$ is an ideal crystal.
\end{proposition}

\begin{proof}
Given $\Lambda=\Gamma+F$, let $L=\{\gamma\in \Gamma: T_\gamma g=g\}$ be the period group of $g$ in $\Gamma$. As $g|_\Gamma$ is fully periodic, $L\subset \Gamma$ is a lattice.  In particular, $L$ has finite index in $\Gamma$. To see the latter statement, write $L+K=G$ for some compact set $K$. We then have
\begin{displaymath}
\Gamma=\Gamma\cap (L+K) \subset  L+K\cap \Gamma\subset \Gamma \ .
\end{displaymath}
For the first inclusion, assume $\ell+k\in \Gamma$ for $\ell\in L\subset \Gamma$ and $k\in K$. Then  $k\in \Gamma-\ell\subset \Gamma-L\subset \Gamma-\Gamma=\Gamma$, and the inclusion follows.
As $K\cap \Gamma$ is a finite set, this shows that $L$ has finite index in $\Gamma$.
We can thus write $\Gamma=L+E$ for a lattice $L$ and a finite set $E$. Now define the finite set
\begin{displaymath}
F_g=\{e+f+ g(e+f): e\in E, f \in F\} \ .
\end{displaymath}
Then we have
\begin{displaymath}
\begin{split}
\Lambda^g&= \{\ell+e+f+g(\ell+e+f) : \ell\in L, e\in E, f\in F\} \\
&= \{\ell+e+f+g(e+f) : \ell\in L, e\in E, f\in F\} \\
&= L+F_g \ .
\end{split}
\end{displaymath}
This shows that $\Lambda^g$ is an ideal crystal.
\qed
\end{proof}

\begin{remark}
If $g|_\Gamma$ is fully periodic, then $g(\Gamma)$ is finite. The converse is true if $g|_\Gamma$ is almost periodic since for $g(\Gamma)$ finite, any $\varepsilon$-almost period is a period for sufficiently small $\varepsilon$ in that case. Sometimes, the case $g(\Gamma)$ finite is referred to as commensurate modulation.
\end{remark}

\begin{remark}
By Remark~\ref{rem:dsd}, we may canonically write $\Lambda^g=P(\Lambda^g)\oplus (D\cap \Lambda^g)$, where $D$ is a relatively compact fundamental domain of the lattice $P(\Lambda^g)$. The lattice $L$ used in the above proof might not coincide with $P(\Lambda^g)$. 
\end{remark}

\begin{remark}\label{rem:comm}
The previous result shows that any commensurate lattice modulation with trivial weight function is an ideal crystal. In this sense, ideal crystals naturally arise from lattices by commensurate modulation.  Conversely, in $G=\RR^d$ it is not too difficult to see that any ideal crystal is a commensurate lattice modulation. 
\end{remark}

\begin{appendix}

\section{Group valued almost periodic functions}\label{app:gvap}

We discuss almost periodicity for functions which take values in a topological abelian group.
This subsumes real or complex-valued almost periodic functions, see \cite{B55,C68,K04} for standard treatments. It also subsumes almost periodic functions with values in Euclidean space \cite{Zai}, see also \cite[Sec.~33.26-27]{HR}.%
\footnote{In that context, let us mention methods of sampling along almost periodic sequences \cite{BHL}, which complement dynamical systems techniques and may also be put into a group setting.}
Our results are natural extensions from the case of complex-valued functions. In view of our applications, we restrict analysis to uniformly continuous and bounded functions.

\subsection{Function space topology}

The topology on our function spaces is induced by the uniform structure of uniform convergence.   For background about topologies on function spaces, see e.g.~Chapters 6 and 7 in \cite{Kel} or Chapters II, III.3 and X.1 in \cite{Bou98}.

\medskip

Let $ G,H$ be LCAG, where $H$ is complete. Consider the set $\mathcal H$ of all maps from $G$ to $H$. For each zero neighborhood $W\subset H$ define
\begin{displaymath}
U_W=\{(h_1, h_2) \in \mathcal{H} \times \mathcal{H}: h_1(x) - h_2(x) \in W \ \ \mbox{for all $x \in G$}\}
\end{displaymath}
and consider  $\mathcal{U} = \{U_W: \text{$W$ is a zero neighborhood in $H$} \}$. The following facts are standard.

\begin{lemma}
\begin{itemize}
\item[(i)] $\mathcal{U}$ is a fundamental set of entourages for a uniformity on $\mathcal{H}$. 
\item[(ii)] $\mathcal{H}$ is complete with respect to the above uniformity.  
\item[(iii)] The subspaces $C(G,H)$ of continuous functions, $C_b(G,H)$ of continuous and bounded functions and $C_\mathsf{u}(G,H)$ of uniformly continuous and bounded functions are closed in $\mathcal{H}$.
\end{itemize}
\end{lemma}
Note that for $H=\mathbb C$ with its usual topology, the topology on the function space $\mathcal H$ induced by $\mathcal{U}$ coincides with the topology induced by the supremum norm.

\subsection{Almost periodic functions}

From now on, we restrict to uniformly continuous and bounded functions.  In that setting, almost periodicity can be described in three equivalent ways.

\medskip

For $f\in C_\mathsf{u}(G,H)$ consider its hull $G_f$, i.e., the closure of the translation orbit of $f$ in the topology of uniform convergence. We write
\begin{displaymath}
G_f = \overline{\{T_tf: t\in G\}} \subset C_\mathsf{u}(G,H) \ ,
\end{displaymath}
where we use the notation $(T_tf)(x)=f(-t+x)$ for function translation.
We say that $f$ is \textit{Bochner almost periodic}, if its hull $G_f$ is compact.

\medskip

We say $f$ is \textit{Bohr almost periodic}, if for every zero neighborhood $W \subset H$ the set
\begin{displaymath}
P_W(f) =\{t\in G :  (T_tf, f)\in U_W\}
\end{displaymath}
of $W$-almost periods of $f$ is relatively dense. Here $A\subset G$ is \textit{relatively dense} in $G$ if there is a compact set $K\subset G$ such that $A+K=G$.

\medskip

Almost periodicity can also be defined using the Bohr compactification $G_\mathsf{b}$ of $G$. This is the dual group of the dual group $(\widehat G)_d$ of $G$, equipped with the discrete topology. See \cite[Sec.~4.7]{Fol} for background. We denote by $i_\mathsf{b}:G\to G_\mathsf{b}$ the canonical injection map.

\begin{proposition}\label{prop:BochBohr}
Consider $f\in C_\mathsf{u}(G,H)$. Then the following are equivalent.
\begin{itemize}
\item[(i)] $f$ is Bochner almost periodic.
\item[(ii)] $f$ is Bohr almost periodic.
\item[(iii)] There exists $f_\mathsf{b}\in C(G_\mathsf{b}, H)$ such that $f=f_\mathsf{b}\circ i_\mathsf{b}$.
\end{itemize}
\end{proposition}

A proof can be given by properly adapting the arguments from the case $H=\mathbb C$, which are given e.g.~in \cite{MoSt}. As arguments based on the Bohr compactification appear throughout the article, we write out the proof $(i) \Leftrightarrow (iii)$ for the reader's convenience, assuming that we have established $(i) \Leftrightarrow (ii)$.

\begin{proof}
\noindent  $(i) \Rightarrow (iii)$: Consider the map $Tf:G\to G_f$ given by $t\mapsto (Tf)(t)=T_tf$. It is a continuous group homomorphism. As
$G_f$ is a compact by assumption, we can use the universal property of the Bohr compactification to infer the existence of a continuous group homomorphism $(Tf)_\mathsf{b}: G_\mathsf{b}\to G_f$ satisfying $Tf=(Tf)_\mathsf{b}\circ i_\mathsf{b}$. Let us consider the evaluation map $\delta: G_f\to H$, which is defined by $\delta(g)=g(0)$ for $g\in G_f$. We then have $f=\delta\circ Tf$. The situation is summarized in the following commutative diagram.
\begin{center}
\begin{tikzcd}
G \arrow{r}{f} \arrow[swap]{d}{i_\mathsf{b}} \arrow{rd}{Tf}& H \\
G_\mathsf{b} \arrow[swap]{r}{(Tf)_\mathsf{b}} & G_f \arrow[swap]{u}{\delta}
\end{tikzcd}
\end{center}
Now define $f_\mathsf{b}: G_\mathsf{b}\to H$ by $f_\mathsf{b}=\delta\circ (Tf)_\mathsf{b}$. Then $f=f_\mathsf{b}\circ i_\mathsf{b}$.

\smallskip

\noindent  $(iii) \Rightarrow  (i)$:  Consider $f_\mathsf{b}\in C(G_\mathsf{b},H)$ such that $f=f_\mathsf{b}\circ i_\mathsf{b}$. Consider any zero neighborhood $W\subset H$ and choose some zero neighborhood $V\subset H$ such that $\overline{V}\subset W$. As $G_\mathsf{b}$ is a compact group, the function $f_\mathsf{b}\in C(G_\mathsf{b},H)$ is Bohr almost periodic and hence also Bochner almost periodic by $(ii) \Rightarrow (i)$. For some finite set $F_\mathsf{b}\subset G_\mathsf{b}$ we thus have
\begin{displaymath}
\{T_{t_\mathsf{b}} f_\mathsf{b}: t_\mathsf{b}\in G_\mathsf{b}\} \subset \bigcup_{t_\mathsf{b}\in F_\mathsf{b}} U_V(T_{t_\mathsf{b}}f_\mathsf{b})\,.
\end{displaymath}
Now observe $U_V(T_{t_\mathsf{b}} f_\mathsf{b})\circ i_\mathsf{b} \subset U_V(T_{t_\mathsf{b}}f_\mathsf{b}\circ i_\mathsf{b})$ and  $T_tf=(T_{i_\mathsf{b}(t)}f_\mathsf{b})\circ i_\mathsf{b}$. Hence the above implies
\begin{displaymath}
\begin{split}
\{T_tf: t\in G\}&=\{T_{i_\mathsf{b}(t)}f_\mathsf{b}: t\in G\} \circ i_\mathsf{b}\subset \bigcup_{t_\mathsf{b}\in F_\mathsf{b}} U_V(T_{t_\mathsf{b}}f_\mathsf{b}) \circ i_\mathsf{b} \\
&\subset \bigcup_{t\in F} U_W(T_{i_\mathsf{b}(t)}f_\mathsf{b}) \circ i_\mathsf{b} \subset \bigcup_{t\in F} U_W(T_{t}f)\,,
\end{split}
\end{displaymath}
where $F\subset G$ is some finite set. For the second inclusion we used that $i_\mathsf{b}(G)\subset G_\mathsf{b}$ is dense in $G_\mathsf{b}$. As $W\subset H$ was an arbitrary zero neighborhood,  this shows total boundedness of the hull. As $C_\mathsf{u}(G,H)$ is complete, this implies compactness of the hull \cite[Ch.~6.32]{Kel}.
\end{proof}

\subsection{Almost periodic functions on subgroups}\label{app:apsg}

Let $L$ be a closed subgroup of an LCAG $G$, and let $H$ be an LCAG. We discuss the relation between almost periodic functions $G\to H$ and $L\to H$. It is advantageous to perform the analysis on the Bohr compactifications $L_\mathsf{b}$ and $G_\mathsf{b}$, since we can then work with continuous functions on compact groups. The natural continuous injection maps lead to the following commutative diagram.

\begin{center}
\begin{tikzcd}
L \arrow{r}{i} \arrow[swap]{d}{i_L} & G \arrow{d}{i_G}\\
L_\mathsf{b} \arrow{r}{i_\mathsf{b}} &G_\mathsf{b}
\end{tikzcd}
\end{center}

\begin{lemma}\label{rem:com}
The above continuous group homomorphism $i_\mathsf{b}$ is uniquely defined.
It is an embedding of $L_\mathsf{b}$ into $G_\mathsf{b}$.
\end{lemma}

\begin{proof}
Consider the following commutative diagram.
\begin{center}
\begin{tikzcd}
\widehat{L} &\arrow[swap]{l}{\widehat{i}}   \widehat{G} \\
({\widehat L})_d \arrow{u}{id}  & \arrow[swap]{u}{id} \arrow[swap]{l}{{\widehat i}_d}({\widehat G})_d
\end{tikzcd}
\end{center}
Here  $\widehat{i}:\widehat G\to\widehat L$ denotes the character restriction map, and $({\widehat L})_d$ resp.~$({\widehat G})_d$ denote the groups $\widehat L$ resp.~$\widehat G$ equipped with the discrete topology.
As $L_\mathsf{b}=\widehat{(\widehat L)_d}$ and   $G_\mathsf{b}=\widehat{(\widehat G)_d}$, dualizing this diagram yields the first part of the claim.
We show that $i_\mathsf{b}=\widehat{\widehat{i}_d}$ is one-to-one: Note first that the character restriction map is onto \cite[Thm.~4.2.14]{Rei2}, which follows from Pontryagin duality $\widehat L \simeq \widehat G/ L_0$, where $L_0\subset \widehat G$ is the annihilator of $L$, compare \cite[Thm.~4.2.21]{Rei2}.
Then $\widehat{i}_d:(\widehat G)_d\to(\widehat L)_d$ is a continuous onto group homomorphism. In particular, $\widehat{i}_d$ has dense range. Now we get by Pontryagin duality that $i_\mathsf{b}=\widehat{\widehat{i}_d}$ is one-to-one.
\qed
\end{proof}

Using Proposition~\ref{prop:BochBohr}, we can derive the following two results.

\begin{lemma} \label{almost-periodic-on-G-is-on-L}
Let $f\in C_\mathsf{u}(G,H)$ be an almost periodic function and let $L$ be a closed subgroup of  $G$. Consider the restriction $f_L$ of $f$ to $L$. Then $f_L\in C_\mathsf{u}(L,H)$ is almost periodic on $L$.
\end{lemma}

\begin{proof}
The almost periodic $f\in C_\mathsf{u}(G,H)$ can be written as $f=F\circ i_G$ for a unique continuous $F:G_\mathsf{b}\to H$. Then $F_L: L_\mathsf{b}\to H$, defined by $F_L=F\circ i_\mathsf{b}$, is continuous. Then $F_L\circ i_L: L\to H$ is almost periodic on $L$ and satisfies $F_L \circ i_L(l)= F\circ i_\mathsf{b}\circ i_L(l)= F\circ i_G\circ i (l)=f(l)$. We have shown $F_L\circ i_L=f_L$, with $f_L$ the restriction of $f$ to $L$.
\qed
\end{proof}

\begin{lemma}\label{almost-periodic-on-L-is-on-G}
Let $L$ be a closed subgroup of $G$. Consider $H=\mathbb R^d$ and let $f_L\in C_\mathsf{u}(L,H)$ be an almost periodic function. Then there exists an almost periodic extension $f\in C_\mathsf{u}(G,H)$ of $f_L$ to $G$.
\end{lemma}

\begin{remark}
The above extension might not be unique.
\end{remark}

\begin{proof}
The almost periodic function $f_L\in C_\mathsf{u}(L,H)$ can be written as $f_L=F_L\circ i_L$ for a unique continuous function $F_L:L_\mathsf{b}\to H$. Consider $A=i_\mathsf{b}(L_\mathsf{b})\subset G_\mathsf{b}$. We can define a continuous function $F_A: A\to H$ by $F_A(g)=F_L(l)$, where $l\in L_\mathsf{b}$ is uniquely determined by $i_\mathsf{b}(l)=g$, as $i_\mathsf{b}: L_\mathsf{b}\to A$ is a homeomorphism, compare Lemma~\ref{rem:com}. As $A$ is closed, we can invoke Tietze's extension theorem \cite[Thm.~A.8.3]{DE} on each component of $F_A$  to infer that $F_A$ admits a continuous extension $F: G_\mathsf{b}\to H$. Hence $f\in C_\mathsf{u}(G,H)$ defined by $f=F\circ i_G$ is an almost periodic extension of $f_L$ to $G$.
\qed
\end{proof}

\end{appendix}

\subsection*{acknowledgments}
J.-Y. Lee would like to thank to Hyeong-Chai Jeong for motivating this work. J.-Y. Lee was supported by NRF grant No. 2019R1I1A3A01060365.
N.~Strungaru was supported by NSERC with grants 03762-2014 and 2020-00038, and he is grateful for the support. We thank the referees for very detailed suggestions for improvement.

\end{document}